\documentclass[12pt,a4paper]{amsart}
\usepackage[utf8x]{inputenc}
\usepackage[all]{xy}
\usepackage{ucs}
\usepackage{amsmath,calrsfs,exscale}
\usepackage{amsfonts}
\usepackage{amssymb}
\usepackage{amsthm}
\usepackage{graphicx}
\usepackage{enumerate}
\usepackage{graphicx,xcolor}
\usepackage{bigints}

\newcommand{\less}{\leqslant}
\newcommand{\gtr}{\geqslant}

\usepackage{indentfirst}
\makeatletter
\@addtoreset{equation}{section}
\makeatother

\newtheorem{statement}{}[section]

\newtheorem{theorem}[statement]{Theorem}
\newtheorem{prop}[statement]{Proposition}

\newtheorem{corol}[statement]{Corollary}
\newtheorem{lemma}[statement]{Lemma}
\newtheorem{rem}{Remark}

\newcommand{\C}{\mathbb C}

\newcommand{\R}{\mathbb R}

\newcommand{\N}{\mathbb N}
\newcommand{\T}{\mathbb T}

\newcommand\e{{\rm e}}
\newcommand{\eps}{\varepsilon}
\newcommand\ind{{\rm 1\kern-.30em I}}

\renewcommand{\Re}{\operatorname{Re}}
\renewcommand{\Im}{\operatorname{Im}}
\newcommand\dis{\displaystyle}

\newcommand{\biindice}[3]%
{%

\begin{array}[t]{c}
{\displaystyle #1}\\
{\scriptstyle #2}\\
{\scriptstyle #3}
\end{array}

}

\begin{document}

%

\title[Eventual Ideal properties of semigroups of operators]{Eventual Ideal properties of the Riemann-Liouville analytic semigroup}
\author{S. GORILIHA*}

\address{I. Alam, Université Libanaise, Faculté des Sciences II, 90656, Fanar-Beyrouth, Liban}
\email{ihabalam@yahoo.fr}

\address{I. Chalendar, Université Gustave Eiffel, LAMA, (UMR 8050), UPEM, UPEC, CNRS, 77454, Marne-la-Vallée, France}
\email{isabelle.chalendar@univ-eiffel.fr}

\address{F. El Chami, Université Libanaise, Faculté des Sciences II, 90656, Fanar-Beyrouth, Liban}
\email{fchami@ul.edu.lb}

\address{E. Fricain, Laboratoire Paul Painlev\'e, Universit\'e de Lille, 59655 Villeneuve d'Ascq C\'edex, France}
\email{emmanuel.fricain@univ-lille.fr}
%
\address{P. Lef\`evre, Laboratoire de Math\'ematiques de Lens (LML), EA 2462, Universit\'e d'Artois, Rue Jean Souvraz S.P. 18, 62307 Lens, France}
\email{pascal.lefevre@univ-artois.fr}

\keywords{Volterra operator, Riemann--Liouville semigroup, fractional integration, nuclear operators, Schatten class}
\thanks{* I. Alam, I. Chalendar,  F. El Chami, E. Fricain and P. Lef\`evre.\\
This work was supported by Labex CEMPI (ANR-11-LABX-0007-01), the project FRONT (ANR-17-CE40 - 0021) and the project CEDRE Th\'eorale 2020-2022}

\subjclass[2010]{47B10, 47B38, 47D03 47G10}

\maketitle

\begin{abstract}
In this paper, we revisit the Riemann--Liouville analytic semigroup. In particular, we completely characterize the membership to the Schatten class $\mathcal S^r$ on $L^2(0,1)$, as well as the membership to the class of nuclear operators on $L^p(0,1)$, where $1\less p<\infty$, and the membership to the ideal of absolutely $r$-summing operators for any $r\gtr 1$. 
\end{abstract}

\section{Introduction} 
Let $X$ be a Banach space and $T\in {\mathcal B}(X)$, which means that  the map $T:X\to X$ is linear and continuous. Then $T$ can always be embedded in the \emph{discrete semigroup} of operators $(T^n)_{n\gtr 0}$. When $T$ is a composition operator on some classical Banach spaces $X$ of holomorphic functions on the open unit disc $\mathbb D$, the asymptotic behaviour of $T^n$ as $n~\to\infty$, for the various possible topologies offered by ${\mathcal B}(X)$, is useful for describing all the composition operators which are isometric or similar to isometries, see \cite{ACKS18,ACKS20}.  

The problem of the embeddability of $T\in{\mathcal B}(X)$ in a strongly continuous semigroup is a natural question and a difficult problem in general. Recall that we say that $T\in {\mathcal B}(X)$ is embeddable in  a strongly continuous semigroup (in short a $C_0$-semigroup) if there exists $(T_t)_{t\gtr 0}\subset {\mathcal B}(X)$ such that $T_0=\mathop{Id}$, $T_1=T$, $T_{s+t}=T_tT_s$ for all $s,t\gtr 0$ and \[\lim_{t\to 0}T_t x=x \mbox{ for all }x\in X.\]   

A sufficient spectral condition which guarantees the embeddability of $T$ is the following: $\sigma(T)\subset \Omega$, where $\Omega$ is a simply connected open set which does not contain $0$, so that it is possible to define an analytic logarithm of $T$ using the Dunford--Riesz functional calculus, and then $T_t=T^t$, $t>0$ defines a $C_0$-semigroup.   

A necessary condition for the embeddability is related to the kernel and the codimension of the closure of the range of $T$ \cite{tania}, namely 
\[\dim\ker (T)\in \{0,\infty\} \mbox{ and }\rm{codim}(\overline{TX})\in \{ 0,\infty\}.\]  
From this, one can deduce that when $X$ is a Hilbert space and $T$ is normal,  a necessary and sufficient condition for the embeddability is that 
$\dim\ker (T)\in \{0,\infty\}$ (see \cite[Chapter 5]{tania}). Moreover, when $X$ is a Hilbert space and $T$ is isometric, a necessary and sufficient condition is that $T$ is unitary or ${\rm{codim}} (TX)=\infty$. 

We focus on  the study of the well-known Volterra operator $V$ on  
spaces $X$ of functions over $(0,1)$. 
More specifically, we shall work on $X=L^p(0,1)$, $1\less p<\infty$, and $X=C_0([0,1])=\{f\in C([0,1])\,|\; f(0)=0\}$, rather than on the space of continuous functions $C([0,1])$: indeed we could work on this latter space with the Volterra operator and the associated semigroup (see below). 
Nevertheless since their range is included in $C_0([0,1])$, it makes more sense to work on $C_0([0,1])$. 
Actually, we would loose the strong continuity of the semigroup in which $V$ can be embedded if we considered $C([0,1])$. 

Recall that the Volterra operator $V$ is defined, for every $f\in X$ and every $x\in [0,1]$, by 
\[
V(f)(x)=\int_0^xf(u)\;du.
\]
Since the spectrum of $V$ is $\{0\}$, the embeddability of $V$ is not clear in view of the sufficient condition for embeddability. 
Fortunately the necessary conditions are satisfied and indeed  the computation  of the powers of $V$ quickly suggests that it embeds into the following  semigroup:
 \begin{equation}\label{eq:vol}
V_t(f)(x)=\dfrac{1}{\Gamma(t)}\int_0^x f(u)(x-u)^{t-1}\;du,  
\end{equation}
where $t>0,\hbox{  }x\in[0,1]$ and $\Gamma$ stands for the classical Euler function.

 The semigroup $(V_t)_{t\gtr 0}$ is called the {\emph{Riemann-Liouville semigroup}}. For $a\in\R$, let $\C_a$ be the right half plane $\{z\in\C\,|\; \Re(z)>a\}$.  In fact, \eqref{eq:vol} makes sense if we replace $t$ by $\xi\in\mathbb C_0$, that is: 
\begin{equation}\label{RL1}\tag{RL1}
	V_\xi (f)(x)=\dfrac{1}{\Gamma(\xi)}\int_0^x f(u)(x-u)^{\xi-1}\;du.
\end{equation}
It can also be expressed as:
\begin{equation}\label{RL2}\tag{RL2}
	\dis V_\xi(f)(x)=\dfrac{x^\xi}{\Gamma(\xi)}\int_0^1f(x\theta)(1-\theta)^{\xi-1}\;d\theta.
\end{equation}
\bigskip

The history of this integral (and semigroup) is very rich and goes back to the beginning of the nineteenth century \cite{Liouville,Riemann}. 
It is linked  to the theory of fractional calculus and ordinary differential equations. Since its introduction, it has been generalized in many ways. See for instance \cite{Diethelm} and the reference therein.  
The aim of this paper is  to revisit this semigroup and exhibit several new properties in connection with classical ideals of operators. In particular, we will see that the membership of $V_\xi$ to most of the natural ideals (nuclear operators on $L^p(0,1)$, Schatten class $\mathcal S^r$ on $L^2(0,1)$ or absolutely $r$-summing operators on $L^p(0,1)$) will depend on $\Re(\xi)$.  More precisely, we will show that for each of these ideals $I$, the semigroup $(V_\xi)_{\xi\in\C_0}$ on $L^p(0,1)$ eventually belongs to $I$, meaning that there is a $\tau_0>0$ (depending on the ideal $I$ and $p$) such that $V_\xi$ belongs to $I$ if and only if $\Re(\xi)>\tau_0$. 
\bigskip

The paper is organized as follows. In Section~\ref{sec:2}, we introduce and recall some useful properties of some ideals of operators we are interested in. We also revisit a three lines theorem in the context of Schatten classes. In Section~\ref{section-semi-group}, we collect some classical properties of the Riemann-Liouville semigroups concerning boundedness and prove again that $(V_\xi)_{\xi\in\C_0}$ is an analytic strongly continuous semigroup. In Section~\ref{sec:unicellularity}, we analyze the property of unicellularity of $V_\xi$, for $\xi\in\C_0$. 
In the last section, which contains the main results, we completely characterize the membership to the class of nuclear operators on $L^p(0,1)$, $p\gtr 1$ as well as the membership to the Schatten class $\mathcal S^r$ on $L^2(0,1)$. Moreover, we characterize those $V_\xi$ which are absolutely $r$-summing on $L^p(0,1)$ for every $r,p\gtr1$. Finally, we characterize those $V_\xi$ which are nuclear, $r$-integral and absolutely $r$-summing on $C([0,1])$. 

Note that Sections 3 and 4 present more or less known results, but in a slightly more general context, and are more a survey with self-contained proofs (which are sometimes maybe a bit different than the classical ones). However Section 5 contains completely new results.

\section{preliminaries}\label{sec:2}
In the following, $X,Y$ are Banach spaces, $\mathcal B(X,Y)$ denotes the space of linear and continuous operators from $X$ to $Y$, and $\mathcal K(X,Y)$ denotes the ideal of compact operators from $X$ to $Y$. For a Banach space $X$, we denote by $X'$ its dual (equipped with its usual norm), and we denote by $B_X$ the closed unit ball of $X$. Finally, for $1\less p\less \infty$, $p'$ denotes its conjugate exponent, that is $\frac{1}{p}+\frac{1}{p'}=1$. 

We recall some definitions and basic facts about well-studied ideals of operators related to compact operators. We refer to the textbooks \cite{DJT,DS,GGK,RS} for further details and references.
 
\subsection{Ideals of linear operators: definitions and main properties}
 \subsubsection{Nuclear operator}
Let $1\less p<\infty$, and let $T\in\mathcal B(X,Y)$. The operator $T$ is called \emph{$p$-nuclear}, and we write $T\in\mathcal N_p(X,Y)$, if there are sequences $(\lambda_n)_n\in\ell^p(\mathbb N)$, $(\varphi_n)_n$ in the unit ball of $X'$ and $(y_n)_n$ in the unit ball of $\ell^{p'}_{weak}(Y)$ such that 
\begin{equation}\label{eq:defn-nuclear}
T=\sum_{n=1}^{\infty}\lambda_n \varphi_n\otimes y_n,
\end{equation}
where $\varphi_n\otimes y_n$ is the rank-one operator in $\mathcal B(X,Y)$ defined, for $x\in X$, by $(\varphi_n\otimes y_n)(x)=\varphi_n(x)y_n$, and 
\[
\ell^{p'}_{weak}(Y)=\{(y_n)_n\in Y^{\mathbb N}:\,\forall \psi\in Y', (\psi(y_n))_n\in\ell^{p'}(\mathbb N)\},
\]
equipped with the norm
\[
\|(y_n)_n\|_{\ell^{p'}_{weak}(Y)}=\sup_{\psi\in B_{Y'}}\|(\psi(y_n))_n\|_{\ell^{p'}}.
\]
When $T$ is $p$-nuclear, we write $\nu_p(T)=\inf\|(\lambda_n)_n\|_{\ell^p}$ where the infimum is taken over all factorizations  \eqref{eq:defn-nuclear}, and this defines its $p$-nuclear. 
Endowed with this norm, the space $(\mathcal N_p(X,Y),\nu_p(\cdot))$ is a Banach space. It is also an ideal, in the sense that if $T\in\mathcal N_p(X,Y)$, $U\in\mathcal B(X_1,X)$ and $V\in\mathcal B(Y,Y_1)$, then $VTU\in\mathcal N_p(X_1,Y_1)$ and $\nu_p(VTU)\less \|V\| \nu_p(T)\|U\|$. Finally, for $1\less p<q<\infty$, we have $\mathcal N_p(X,Y)\subset \mathcal N_q(X,Y)\subset \mathcal K(X,Y)$ and for every $T\in \mathcal N_p(X,Y)$, $\nu_q(T)\less \nu_p(T)$ (see \cite[page 113]{DJT}). 
The particular class of $1$-nuclear operators is simply called the class of \emph{nuclear operators}. Note that if $p=1$, then $p'=\infty$ and 
\[
\|(y_n)_n\|_{\ell^{p'}_{weak}(Y)}=\sup_{\psi\in B_{Y'}}\sup_{n\gtr 1}|\psi(y_n)|=\sup_{n\gtr 1}\|y_n\|_Y.
\]
Therefore an operator $T\in \mathcal B(X,Y)$ is nuclear if and only if there exist $(\varphi_n)_n\subset X'$ and $(y_n)_n\subset Y$ such that 
\[
T=\sum_{n=1}^{\infty}\varphi_n\otimes y_n,
\]
and 
\[
\sum_{n=1}^\infty \|\varphi_n\|_{X'}\|y_n\|_Y<\infty. 
\]
 \\

\subsubsection{Absolutely $p$-summing operators}\label{subsubsection-summing-operator}

Let $1\less p<\infty$ and let $T\in\mathcal B(X,Y)$. The operator $T$ is called \emph{absolutely $p$-summing}, and we write $T\in \Pi_p (X,Y)$, if there exists  $\rho>0$ such that for every $n\in\mathbb N$ and every $x_1,\cdots, x_n\in X$ we have
\begin{equation}\label{eq:psum}  
  \left( \sum_{i=1}^n\|T x_i\|_Y^p\right)^{1/p}\less \rho \sup_{\varphi\in B_{X'}}  \left( \sum_{i=1}^n|\varphi( x_i)|^p\right)^{1/p}.
 \end{equation}
For $T\in \Pi_p (X,Y)$ we define its absolutely $p$-summing norm $\pi_p(T)$ by 
\[  \pi_p(T)=\inf\{  \rho>0:\eqref{eq:psum}\mbox{ holds}\}\]
The space $(\Pi_p(X,Y),\pi_p(\cdot))$ is a Banach space. It is also an ideal, in the sense that if $T\in\Pi_p(X,Y)$, $U\in\mathcal B(X_1,X)$ and $V\in\mathcal B(Y,Y_1)$, then $VTU\in\Pi_p(X_1,Y_1)$ and $\pi_p(VTU)\less \|V\| \pi_p(T)\|U\|$. In addition, we have
\begin{equation}\label{Diestel-eq2}
1\less p<q<\infty\implies \Pi_p(X,Y)\subset  \Pi_q(X,Y),
\end{equation}
and for every $T\in\Pi_p(X,Y)$, $\pi_q(T)\less \pi_p(T)$.
\\

The most simple example of a $p$-absolutely summing operator is given as follows. 
Let $\mu$ be a probability Borel measure on a compact space $K$ and let $1\less p\less \infty$. Then the formal identity 
\begin{equation}\label{Diestel-eq3}
i_p: C(K)\to L^p(\mu) \mbox{ is absolutely $p$-summing}, 
\end{equation}
and $\pi_p(i_p)=1$. See \cite[page 201]{Woj} or \cite[page 40]{DJT}. It is also known that every absolutely $p$-summing operator, $1\less p<\infty$, is weakly compact and completely continuous. Moreover, if $T\in \Pi_p(X,Y)$ and $S\in\Pi_q(Y,Z)$, then $ST$ is compact \cite[Page 50]{DJT}. 

A famous result of Pietsch states that the previous example of absolutely $p$-summing operators is in some sense canonical. We recall here the following version which we shall use when studying the membership to the absolutely $p$-summing operators on $C([0,1])$. See \cite[Theorem 2.12 and Corollary 2.15]{DJT}.
\begin{theorem}[Pietsch]\label{Pietsch-theorem}
Let $K$ be a compact space and $T:C(K)\to Y$ be a bounded operator. Then $T$ is absolutely $p$-summing if and only if there exists a regular Borel probability measure $\nu$ on $K$ and a constant $C>0$ such that, for every $f\in C(K)$, we have  
\begin{equation}\label{eq:pietsch}
\|Tf\|_Y\leq C \left(\int_K |f(x)|^p\,d\nu(x)\right)^{\frac{1}{p}}.
\end{equation}
Moreover, in this case, the infimum of $C$ satisfying \eqref{eq:pietsch} is $\pi_p(T)$. 
\end{theorem}

Finally, let us mention the following well-known result of Grothendieck  \cite[Theorem 3.4]{DJT}.
\begin{theorem}[Grothendieck]\label{Diestel-theorem4}
Let $\mu$ and $\nu$ be any measures and let $T:L^1(\mu)\to L^2(\mu)$ be any bounded operator. Then $T$ is absolutely $1$-summing and  
\[
\pi_1(T)\less K_G \|T\|,
\] 
where $K_G$ is a universal constant called the Grothendieck constant.
\end{theorem}


It is well known that if $(\Omega,\Sigma,\mu)$ is any measure space and $1\less p<\infty$, then $L^p(\mu)$ has cotype $\max(p,2)$ (see \cite{DJT,Woj} for more details on the cotype). Moreover, every Hilbert space have cotype $2$. See \cite[Corollary 11.7 and Corollary 11.8]{DJT}. We shall need the following links with the ideal of absolutely summing operators. If $Y$ has cotype $2$, then for any Banach space $X$ and all $2<r<\infty$, we have $\Pi_r(X,Y)=\Pi_2(X,Y)$. Furthermore, if $X$ and $Y$ both have cotype $2$, then  
\begin{equation}\label{Diestel-eq5}
\Pi_r(X,Y)=\Pi_1(X,Y) \quad \mbox{for all } 1\less r<\infty.
\end{equation} 
See \cite[Theorem 11.13 and Corollary 11.16]{DJT}. In particular, if $H$ and $K$ are Hilbert spaces, then
\begin{equation}\label{Diestel-eq6}
\Pi_r(H,K)=\Pi_1(H,K) \quad \mbox{for all } 1\less r<\infty.
\end{equation}  
\subsubsection{$p$-integral operators}
 Let $1\less p\less \infty$ and let $T\in\mathcal B(X,Y)$. The operator $T$ is called a \emph{$p$-integral operator}, and we write $T\in {\mathcal I}_p(X,Y)$, if there exist a probability measure $\mu$ and two  bounded linear operators $U:L^p(\mu)\to Y''$ and $V:X\to L^\infty(\mu)$ ($Y''$ is the bidual space of $Y$) giving rise to the commutative diagram

 \def\commutatif{\ar@{}[rrd]|{\LARGE{\circlearrowleft}}}
 \[ 
\xymatrix{
 X \ar[r]^T \ar[d]_V \commutatif & Y\ar[r]^{k_Y} &Y''\\
 L^\infty(\mu)\ar[rr]_{i_\mu} &  & L^p(\mu)\ar[u]_U
}
\]
where $i_\mu$ is the formal identity and $k_Y$ is the canonical isometric embedding.  With each $T\in {\mathcal I}_p(X,Y)$ we associate its $p$-integral norm 
  \[ \iota_p(T)=\inf \|U\| \|V\|,\]
  where the infimum is taken over all measures $\mu$ and $U,V$ as above. It is known that $({\mathcal I}_p(X,Y),\iota_p(\cdot))$ is a Banach space. It  is also an ideal, in the sense that if $T\in\mathcal I_p(X,Y)$, $U\in\mathcal B(X_1,X)$ and $V\in\mathcal B(Y,Y_1)$, then $VTU\in\mathcal I_p(X_1,Y_1)$ and $\iota_p(VTU)\less \|V\| \iota_p(T)\|U\|$. 
Finally, if $1\less p<q<\infty$, then we have  $\mathcal I_p(X,Y)\subset  \mathcal I_q(X,Y)$, and for every $T\in\mathcal I_p(X,Y)$, $\iota_q(T)\less \iota_p(T)$.
\\
The formal identity 
\begin{equation}\label{Diestel-eq345EE34}
i_p: C(K)\to L^p(\mu) \mbox{ is absolutely $p$-integral},
\end{equation}
where $\mu$ is a finite Borel measure on a compact space $K$ and $1\less p\less \infty$. 

\subsubsection{Schatten classes}
 If $H,K$ are Hilbert spaces and $T\in\mathcal K(H,K)$, we denote by $(s_n(T))_n$ the decreasing sequence of the singular values of $T$ (i.e. the eigenvalues of $\sqrt{T^*T}$ where $T^*$ is the adjoint operator of $T$). For $1\less p\less \infty$, we say that $T$ belongs to the \emph{Schatten class operator} ${\mathcal S}^p(H,K)$ if $(s_n(T))_n\in \ell^p(\mathbb N)$. Moreover, if we set 
\[
\|T\|_{\mathcal S^p}=\|(s_n(T))_n\|_{\ell^p},
\]
then $({\mathcal S}^p(H,K),\|\cdot\|_{{\mathcal S}^p})$ is a Banach space and it is an ideal, in the sense that if $T\in {\mathcal S}^p(H,K)$, $U\in \mathcal B(H_1,H)$, $V\in\mathcal B(K,K_1)$, then $VTU\in\mathcal S^p(H_1,K_1)$, and we have 
\begin{equation}\label{eq:Schatten-ideal}
\|VTU\|_{\mathcal S^p}\less \|V\| \|T\|_{\mathcal S^p}\|U\|.
\end{equation}
Note that $\mathcal S^{\infty}(H,K)=\mathcal K(H,K)$ and, for $T\in\mathcal S^{\infty}(H,K)$, we have $\|T\|_{\mathcal S^\infty}=\|T\|$. Moreover, ${\mathcal S}^2$ is called the \emph{Hilbert--Schmidt class operator}, whereas ${\mathcal S}^1$ is called the \emph{trace class operator}. For $1\less p\less q\less \infty$, we have  the following inclusions
\[  
{\mathcal S}^1(H,K)\subset {\mathcal S}^p(H,K)\subset {\mathcal S^q}(H,K)\subset {\mathcal K}(H,K),
\]
and for every $T\in {\mathcal S}^p(H,K)$, we have $\|T\|_{\mathcal S^q}\less \|T\|_{\mathcal S^p}$. Finally, we recall that if $T\in\mathcal B(H,K)$ and $1\less p<\infty$, then $T$ belongs to $\mathcal S^p(H,K)$ if and only if for every orthonormal sequences $\big(h_n\big)_{n\gtr1}$ of $H$ and $\big(k_n\big)_{n\gtr1}$ of $K$, we have
\begin{equation}\label{cnschatten1}
\sum_{n\gtr1}\big|\langle T(h_n),k_n\rangle\big|^p<\infty\;.
\end{equation}
Moreover
\begin{equation}\label{cnschatten2}
\|T\|_{\mathcal{S}^{p}}=\sup\Big(\sum_{n\gtr1}\big|\langle T(h_n),k_n\rangle\big|^p\Big)^{\frac{1}{p}}
\end{equation}
where the supremum runs over all the orthonormal sequences $\big(h_n\big)_{n\gtr1}$ in $H$ and $\big(k_n\big)_{n\gtr1}$ in $K$. See \cite[Theorem~4.6.b, p.~81]{DJT} and its proof.
\\

We shall need a general result making use of the analycity in our framework.

\begin{theorem}\label{interpolschatten}
Let $1\less p_1<p_0\less \infty$, $\alpha_0<\alpha_1$, and consider the open set $\dis\Omega=\{z\in\C:\alpha_0<\Re(z)<\alpha_1\}$. Let $(T_z)_{z\in\overline{\Omega}}$ be a family in $\mathcal B(H,K)$ such that for every $(x,y)\in H\times K$, the map $z\mapsto\langle T_z(x),y\rangle$ is bounded and continuous on $\overline{\Omega}$, and holomorphic on $\Omega$. Assume that 
\begin{enumerate}[(i)]
\item $T_z\in\mathcal{S}^{p_1}(H,K)$ for $\Re(z)=\alpha_1$ with $\dis\sup_{\Re(z)=\alpha_1}\big\|T_z\big\|_{\mathcal{S}^{p_1}}<\infty$;
\item $T_z\in\mathcal{S}^{p_0}(H,K)$ for $\Re(z)=\alpha_0$ with $\dis\sup_{\Re(z)=\alpha_0}\big\|T_z\big\|_{\mathcal{S}^{p_0}}<\infty$.
\end{enumerate}
Then $T_z\in\mathcal{S}^{p}(H,K)$ when $\Re(z)=\alpha\in(\alpha_0,\alpha_1)$ with $\alpha=\theta\alpha_1+(1-\theta)\alpha_0$ and  $p\in(p_1,p_0)$ defined by  $\dfrac{1}{p}=\dfrac{\theta}{p_1}+\dfrac{1-\theta}{p_0}\cdot$
\end{theorem}
After this work was completed, we discovered that this result is not new and actually known for a long time: see \cite[Th.13.1, p.137]{GK69}. Nevertheless our presentation here is slightly different, even though it also relies on the three lines theorem.

This theorem is particularly interesting when $T_z=\psi(z)S_z$ where  $(S_z)_{z\in\C_0}$ is an analytic semigroup of operators and $\psi$ is some (non vanishing) holomorphic function on the right half plane. We shall exploit it in this context.
\begin{proof}
We give the proof for a finite $p_0$ but it is easy to adapt the proof when $p_0=\infty$.
Fix $N\in\N$, $a=(a_n)_{n\in\N}$ in the unit ball of $\ell^{p'}$ and two arbitrary orthonormal sequences $(h_n)_{n\in\N}$ of $H$ and $(k_n)_{n\in\N}$ of $K$. Write $a_n=u_n|a_n|$ with $|u_n|=1$, and, for $j\in\{0;1\}$,  denote by $S_j=\sup_{\Re(z)=\alpha_j}\big\|T_z\big\|_{\mathcal{S}^{p_j}}$. Let us now introduce  the function
$$z\in\overline{\Omega_0}\longmapsto\Phi(z)=\sum_{n=0}^Nu_n|a_n|^{p'\rho(z)+i\beta}\langle T_{z\alpha_1+(1-z)\alpha_0}(h_n),k_n\rangle $$
where $\Omega_0=\{z\in\C:0<\Re(z)<1\}$, $\beta\in\R$ and 
\[
\dis\rho(z)=1-\dfrac{z}{p_1}-\dfrac{1-z}{p_0}=\dfrac{z}{p'_1}+\dfrac{1-z}{p'_0}\cdot
\]
Note that for $z\in{\Omega_0}$, $z\alpha_1+(1-z)\alpha_0\in\Omega$ and if $z\in\partial{\Omega_0}$, then $z\alpha_1+(1-z)\alpha_0\in\partial\Omega$. The function $\Phi$ is holomorphic on ${\Omega_0}$ and moreover it is also clearly bounded on $\overline{\Omega_0}$ (with a bound maybe depending on $N$). 
Actually, on the boundary of ${\Omega_0}$, we have some estimates not depending on $N$.
Indeed, from H\"older's inequality and the fact that $\dis\|a\|_{\ell^{p'}}\less 1$, we have, for $j\in\{0;1\}$ and every $\sigma\in\R$,
\begin{align*}
|\Phi(j+i\sigma)|&\le\sum_{n=0}^N|a_n|^{\frac{p'}{p'_j}}|\langle T_{w_j}(h_n),k_n\rangle|\\
&\le\Big(\sum_{n=0}^N|\langle T_{w_j}(h_n),k_n\rangle|^{p_j}\Big)^{\frac{1}{p_j}}\\
&\le \|T_{w_j}\|_{\mathcal{S}^{p_j}}
\end{align*}
where $w_j=j\alpha_1+(1-j)\alpha_0+i\sigma(\alpha_1-\alpha_0)$. Since $\Re(w_j)=\alpha_j$, we get 
\[
\sup_{\Re(z)=j}|\Phi(z)|\le S_j\;.
\]
From the three lines theorem (see \cite{Had} for instance), we get
$$\forall \theta\in [0,1], \forall \sigma\in\R\,,\qquad|\Phi(\theta+i\sigma)|\le S_0^{1-\theta} S_1^{\theta}\;. $$
For any $\sigma\in\R$, we choose in the sequel $\beta=p'\sigma\Big(\dfrac{1}{p_1}-\dfrac{1}{p_0}\Big)$ so that we check that $p'\rho(\theta+i\sigma)+i\beta=1$ and for $z=\theta+i\sigma$, $z\alpha_1+(1-z)\alpha_0=\alpha+i\sigma(\alpha_1-\alpha_0)$. Thus, for every $\sigma\in\R$, we get 
$$\Bigg|\sum_{n=0}^Na_n\langle T_{\alpha+i\sigma(\alpha_1-\alpha_0)}(h_n),k_n\rangle\Bigg|\le S_0^{1-\theta} S_1^{\theta}\;. $$
Taking the supremum over $a$ in the unit ball of $\ell^{p'}$ and then on $N\in\N$, we get 
$$\Big(\sum_{n=0}^{+\infty}|\langle T_{\alpha+i\sigma(\alpha_1-\alpha_0)}(h_n),k_n\rangle|^{p}\Big)^{\frac{1}{p}}\le S_0^{1-\theta} S_1^{\theta}\;. $$
Finally, taking the supremum over all the orthonormal sequences (recall \eqref{cnschatten2}), we have
$$ \big\|T_{\alpha+i\sigma(\alpha_1-\alpha_0)}\big\|_{\mathcal{S}^p} \le S_0^{1-\theta} S_1^{\theta}$$
and, since $\sigma$ is arbitrary, we conclude that
\begin{equation}\label{interpolestimate}
\sup_{\Re(z)=\alpha}\big\|T_z\big\|_{\mathcal{S}^p} \le S_0^{1-\theta} S_1^{\theta}\;.
\end{equation}
\end{proof} 
 We conclude this subsection on Schatten class operators by a nice factorization result. Let $1\less p,q,r<\infty$ such that $\frac{1}{p}+\frac{1}{q}=\frac{1}{r}$. An operator $T\in\mathcal B(H_1,H_2)$ belongs to $\mathcal S^r(H_1,H_2)$ if and only if there exist a Hilbert space $H$ and two bounded operators $V\in\mathcal S^p(H,H_2)$ and $W\in\mathcal S^q(H_1,H)$ such that $T=VW$. In such a case, we have
 \begin{equation}\label{eq:S2-S2=S1}
\|T\|_{\mathcal S^r}=\inf \|V\|_{\mathcal S^p}\|W\|_{\mathcal S^q},
 \end{equation}
 where the infimum is taken over all such factorizations. Note that the infimum is actually a minimum. See \cite[Theorem 6.3]{DJT}.
 \subsubsection{Order bounded operator}
Let $(\Omega,\Sigma,\mu)$ be any measure space and let $1\less p<\infty$. A non-empty subset $M$ of $L^p(\mu)$ is said to be \emph{order bounded} if there exists a non-negative function $h\in L^p(\mu)$ such that 
\[
|f|\less h\quad \mbox{$\mu$-almost everywhere},
\]
for each $f\in M$. A Banach space operator $T:X\to L^p(\mu)$ is called \emph{order bounded} if $T(B_X)$ is an order bounded subset of $L^p(\mu)$. Recall that $B_X$ denotes the closed unit ball of $X$. 
  
 \subsubsection{Links between all these classes}\label{subsection:links-ideals}
 Let $X,Y$ be Banach spaces and let $1\less p<\infty$. We have
 \begin{equation}\label{Diestel-eq7}
 \mathcal N_p(X,Y)\subset \mathcal I_p(X,Y)\subset \Pi_p(X,Y),
 \end{equation}
 and both inclusions are contractive \cite[pages 97 and 113]{DJT}. Moreover, we have $\mathcal I_2(X,Y)=\Pi_2(X,Y)$ with equality of norms and furthermore, if $Y$ is a subspace of an $L^p$-space, $1\less p\less 2$, then for every Banach space $X$ and for every $2\less q<\infty$, we have
 \[
 \Pi_q(X,Y)=\mathcal I_q(X,Y)=\mathcal I_2(X,Y).
 \]
 See \cite[Page 99]{DJT}.  When $H$ and $K$ are Hilbert spaces, we can say more about the relations between all these ideals. If $1<p<\infty$, then 
 \begin{equation}\label{Diestel-eq10}
 \mathcal I_p(H,K)=\mathcal N_p(H,K)=\Pi_p(H,K)=\Pi_2(H,K)=\mathcal S^2(H,K)
 \end{equation}
 isomorphically, and even isometrically if $p=2$. Moreover, 
 \[
\mathcal I_1(H,K)=\mathcal N_1(H,K)=\mathcal S^1(H,K)
\]
isometrically. See \cite[Corollary 3.16, Theorem 4.10 and Theorem 5.30]{DJT}.\\
 
Regarding the property of order boundedness, we shall need the following known facts that we gather in a theorem. In particular, we see that order boundedness yields an important characterization of Banach space operators whose duals are absolutely $p$-summing. 
 \begin{theorem}\label{Diestel-theorem8}
Let $1\less p<\infty$. 
\begin{enumerate}[(i)]
\item  Let $\mu$ be any measure. Then an order bounded operator $T:X\to L^p(\mu)$ is $p$-integral.
\item A Banach space operator $T:X\to Y$  has an adjoint which is absolutely $p$-summing if and only if, however we choose a measure $\mu$ and a (bounded) operator $U:Y\to L^p(\mu)$, the composition $UT:X\to L^p(\mu)$ is order bounded.
\item Let $\mu$ be any measure. If the operator $T:X\to L^p(\mu)$ has an adjoint which is absolutely $p$-summing, then $T$ must be order bounded.
\end{enumerate}
 \end{theorem}
 \begin{proof}
The assertion $(i)$ is proved in \cite[Proposition 5.18]{DJT}. The assertion $(ii)$ is proved in  \cite[Theorem 5.20]{DJT}. The assertion $(iii)$ follows immediately from $(ii)$. 
\end{proof}

 We shall also need the following facts on composition between these different ideals, that we also gather in the following result. 
 \begin{theorem}\label{Diestel-theorem9}
  Let $X,Y,Z$ be Banach spaces, let $U\in \mathcal B(Y,Z)$ and $V\in \mathcal B(X,Y)$ and let $1\less p,q,r<\infty$ satisfy $\frac{1}{r}=\frac{1}{p}+\frac{1}{q}$.
 \begin{enumerate}[(i)]
 \item If $U$ is compact and $V$ is a $p$-integral operator, or if $U$ is a $p$-integral operator and $V$ is compact, then the product $UV$ is a $p$-nuclear operator. 
\item If $U$ is an absolutely $p$-summing operator and $V$ is a $q$-nuclear operator, or if $U$ is a $p$-nuclear operator and $V$ is an absolutely $q$-summing operator, then $UV$ is a $r$-nuclear operator. 
\item  If $U$ and $V$ are absolutely $2$-summing operators, then $UV$ is a nuclear operator.
\end{enumerate}
 \end{theorem}
 \begin{proof}
 The assertion $(i)$ is proved in \cite[Theorems 5.27 and 5.28]{DJT}. The assertion $(ii)$  follows from  \cite[Theorem 5.29]{DJT}. The assertion $(iii)$ is contained  in \cite[Theorem 5.31]{DJT}. 
  \end{proof}
 
 \subsection{Convolution on $L^p(0,1)$} 
 In this subsection, we briefly recall some standard facts on the convolution on $L^p(0,1)$. 
 Let $1\less p<\infty$ and $p'$ its conjugate exponent. For every  $f\in L^p(0,1)$ and $g\in L^{p'}(0,1)$, we define, for every $x\in [0,1]$, 

 \[
(f\ast g)(x)=\int_0^x f(t)g(x-t)\,dt .
 \]
It is well known that $f\ast g\in C_0([0,1])$. Moreover, thanks to H\"older's inequality, we have $\|f\ast g\|_\infty \less \|f\|_p \|g\|_{p'}$. 
More generally, recall the Young's convolution inequality. Let $1\less p,q,r\less \infty$ satisfying $\frac{1}{p}+\frac{1}{q}=\frac{1}{r}+1$ and let $f\in L^p(0,1)$, $g\in L^q(0,1)$. Then the convolution $(f\ast g)(x)$ is defined (as above) for almost all $x\in [0,1]$ and  $f\ast g\in L^r(0,1)$, with $\|f\ast g\|_r\less \|f\|_p\|g\|_q$. 

When studying the nuclearity of $V_\xi$ on $C([0,1])$, we shall need the following result of factorization.
\begin{theorem}[Salem]\label{thm:salem}
Let $f\in C([0,1])$ with $f(0)=f(1)$ and let $\delta>0$. Then there exists $g\in C([0,1])$ with $g(0)=g(1)$ and $h\in L^1(0,1)$ such that 
\[
f=g\ast h\quad \mbox{and }\quad \|g-f\|_\infty\leq \delta \mbox{ and }\|h\|_1=1.
\]
\end{theorem}
\begin{proof}
See \cite[Theorem 32.31]{Hewitt-Ross} or \cite[Exercice 3.1, page 70]{Katz}. Note that the result in \cite{Hewitt-Ross} is stated and proved in the general context of a compact group. Using the usual identification of functions continuous on the unit circle $\mathbb T$ with the functions  $f\in C([0,1])$ satisfying $f(0)=f(1)$, we obtain this version.
\end{proof}

In Section 3, we revisit the unicellularity of the Riemann-Liouville semigroup and the method is based on the following classical result of Titchmarsh that we recall now.  See \cite{Titch}, or \cite{Mikusinski} for a simpler proof.
\begin{theorem}[Titchmarsh]\label{thm:titchmarsh}
Let $F$ and $G$ be two functions in $L^1(0,1)$. If $F\ast G=0$ almost everywhere on $(0,1)$, then there exists a number $\alpha$ $(0\less \alpha\less 1)$ such that $F$ and $G$ are equal to zero almost everywhere on the intervals $[0,\alpha]$ and $[0,1-\alpha]$ respectively.
\end{theorem}
More precisely, we shall need the following consequence that we prove now.
\begin{corol}\label{Cor:titch}
Let $\xi\in\mathbb C_0$, let $f\in L^p(0,1)$ and let $g\in L^{p'}(0,1)$ where $1\less p<\infty$ and $p'$ its conjugate exponent. Assume that for every $n\gtr 1$, we have 
\[
\int_0^1 g(x)\left(\int_0^x f(s)(x-s)^{n\xi-1}\,ds\right)\,dx=0.
\]
Then there exists  $\alpha\in [0,1] $ such that $f(s)=0$ for almost all $s\in [0,\alpha]$ and $g(s)=0$ for almost all $s\in [\alpha,1]$.
\end{corol}
\begin{proof}
Using the change of variable $u=x-s$ in the inner integral, the hypothesis can be rewritten as 
\[
\int_0^1 g(x)\left(\int_0^x f(x-u)u^{n\xi-1}\,du\right)\,dx=0,\qquad n\gtr 1.
\]
We apply now the Fubini theorem to get 
\begin{equation}\label{eq:Titchmarsh}
\int_0^1 u^{n\xi-1}\left(\int_u^1 g(x)f(x-u)\,dx\right)\,du=0,\qquad n\gtr 1.
\end{equation}
Denote by $h$ the function defined by $h(u)=\int_u^1 g(x)f(x-u)\,dx$, $u\in [0,1]$. 
It is standard that this defines a continuous function on $[0,1]$ and in particular $h$ belongs to $L^2(0,1)$.  
 
Consider an integer $n_0$ such that $n_0>(2\Re(\xi))^{-1}$. If $\lambda_n:=n\xi-1$, we have for every $n\gtr n_0$, $\lambda_n\in\mathbb C_{-1/2}$ and since 
\[
\frac{\frac{1}{2}+\Re(\lambda_n)}{\left|\lambda_n+\frac{1}{2}\right|^2+1}\sim \frac{\Re(\xi)}{n|\xi|^2}, \quad \mbox{as }n\to\infty,
\]
it follows that 
\[
\sum_{n\gtr n_0}\frac{\frac{1}{2}+\Re(\lambda_n)}{\left|\lambda_n+\frac{1}{2}\right|^2+1}=+\infty.
\]
Hence, by the M\"untz theorem, the system $\{u^{n\xi-1}:n\gtr n_0\}$ is complete in $L^2(0,1)$ (see for instance \cite{FL}). 
Thus, we get from \eqref{eq:Titchmarsh} that, for every $u\in [0,1]$, $h(u)=0$. In particular, for every $u\in [0,1]$, we have 
\[
0=h(1-u)=\int_{1-u}^1 g(x)f(x-1+u)\,dx=\int_0^u g(1-v)f(u-v)\,dv.
\]
By the Titchmarsh theorem, there exists $\alpha\in [0,1]$ such that $f(v)=0$ for almost all $v\in [0,\alpha]$ and $g(1-v)=0$ for almost all $v\in [0,1-\alpha]$, which gives the conclusion.
\end{proof}

\section{The Riemann-Liouville semigroup}\label{section-semi-group} 

We now present and prove the main properties of the Riemann-Liouville semigroup. Some of them are well-known but some short proofs are provided for self-completeness.
First of all, we confirm that the definition given in the introduction makes sense. 
Indeed, for $\xi\in\mathbb C_0$, let us introduce the function $\phi_\xi$ defined by 
\begin{equation}\label{phi}
\phi_\xi(u)=\frac{1}{\Gamma(\xi)}u^{\xi-1},\qquad u\in (0,1).
 \end{equation}
Clearly $\phi_\xi\in L^1(0,1)$ with norm $\dfrac{1}{\Re(\xi)|\Gamma(\xi)|}\,\cdot$ Therefore the operator $V_\xi$ is actually defined as $V_\xi f=f\ast \phi_\xi$ for any $f\in L^1(0,1)$ so that the formulas \eqref{RL1} and \eqref{RL2} given in the introduction make sense for almost every $x\in(0,1)$ for any $f\in L^p(0,1)\subset L^1(0,1)$.
We can now present some results on the boundedness of $V_\xi:L^p(0,1)\to L^q(0,1)$.

\begin{prop}\label{bound}\
Let $1\less p\less q<\infty$.
\begin{enumerate}[(i)]
\item If $\xi\in\mathbb C_{\frac{1}{p}-\frac{1}{q}}$, then $V_\xi:L^p(0,1)\to L^q(0,1)$ is bounded and we have
\begin{equation}\label{normVtpq}
\|V_\xi\|_{\mathcal B(L^p,L^q)}\less\dfrac{1}{|\Gamma(\xi)|}\Bigg(\dfrac{\frac{1}{p'}+\frac{1}{q}}{\Re(\xi)-\frac{1}{p}+\frac{1}{q}}\Bigg)^{\frac{1}{p'}+\frac{1}{q}}\cdot
\end{equation}
\item In particular, for every $p\gtr1$, and for every $\xi\in\C_0$, the operator $V_\xi$ is bounded from $L^p(0,1)$ to itself and we have 
\begin{equation}\label{normVt0}
\|V_\xi\|_{\mathcal B(L^p)}\less\dfrac{1}{\Re(\xi)|\Gamma(\xi)|}\cdot
\end{equation}

\item When $1< p< q<\infty$ and $\Re(\xi)=\frac{1}{p}-\frac{1}{q}\,$, then $V_\xi$ is bounded from $L^p(0,1)$ to $L^q(0,1)$ and 
\begin{equation}\label{normVtHL}
\|V_\xi\|_{\mathcal B(L^p,L^q)}\less K_{p,q}\frac{|\Gamma(\Re(\xi))|}{|\Gamma(\xi)|}\cdot
\end{equation}
where $K_{p,q}$ depends on $p$ and $q$ only.

\item If $p>1$ and $\xi\in\mathbb C_{\frac{1}{p}}$, then $V_\xi:L^p(0,1)\to C_0([0,1])$ is bounded and we have
\begin{equation}\label{eq:sddqdsqd02343EDSDCWl}
\|V_\xi\|_{\mathcal B(L^p,C_0)}\less \frac{1}{|\Gamma(\xi)|}\frac{1}{((\Re(\xi)-1)p'+1)^{1/p'}}\,\cdot
\end{equation}

\item If $\Re(\xi)\gtr 1$, then  $V_\xi:L^1(0,1)\to C_0([0,1])$ is bounded and we have
\[
\|V_\xi\|_{\mathcal B(L^1,C_0)}\less \frac{1}{|\Gamma(\xi)|}\,\cdot
\]

\item Let $\xi\in\C_0$. Then  $V_\xi:L^\infty(0,1)\to C_0([0,1])$ is bounded and we have 
$$\|V_\xi\|_{\mathcal B(L^\infty,C_0)}\less \frac{1}{\Re(\xi)|\Gamma(\xi)|}\,\cdot$$
\end{enumerate}
\end{prop}
As usual, in the formulas with (conjugate) exponents, $\dfrac{1}{\infty}$ means $0$.

\begin{proof}
Denote by $\tau=\Re(\xi)$.  First we observe that, for $1\less r<\infty$, $\phi_\xi\in L^r(0,1)$ if and only if $\xi\in\C_{\frac{1}{r'}}$ and, in this case,
\begin{equation}\label{eq:sddsssdqsdqsddnc923E0}
\|\phi_\xi\|_{r}=\frac{1}{|\Gamma(\xi)|}\frac{1}{((\tau-1)r+1)^{\frac{1}{r}}}=\frac{r^{-\frac{1}{r}}}{|\Gamma(\xi)|\big(\tau-\frac{1}{r'}\big)^{\frac{1}{r}}}\,\cdot
\end{equation}

$(i)$ Let $f\in L^p(0,1)$ and $r=\Big(\dfrac{1}{p'}+\dfrac{1}{q}\Big)^{-1}$. We have $\dfrac{1}{q}=\dfrac{1}{p}+\dfrac{1}{r}-1$, and in particular $r\gtr 1$ since $q\gtr p$. 

Moreover, since $\xi\in\C_{\frac{1}{r'}}$, $\phi_\xi\in L^r(0,1)$ and according to the Young inequalities, we know that $V_\xi f \in L^q(0,1)$ with
$$\|V_\xi f\|_q\less\|\phi_\xi\|_{r}\|f\|_p\,.$$
We get \eqref{normVtpq}.

$(ii) $ is now obvious by definition of conjugate exponents.

$(iii)$ Hardy and Littlewood  proved in \cite[Theorem 4]{HL} that if $\tau=\frac{1}{p}-\frac{1}{q}$, with $1<p<q<\infty$, then, for every function $f$ in $L^p(0,1)$, we have 
\begin{equation}\label{inegalite-HL}
\left(\int_0^1 |V_\tau(f)(x)|^q\,dx\right)^{1/q}\less K \left(\int_0^1|f(x)|^p\,dx\right)^{1/p},
\end{equation}
where $K=K(p,q)$ depends only on $p$ and $q$. Observe now that, for $\Re(\xi)=\tau= \frac{1}{p}-\frac{1}{q}$, we have
\[
|V_\xi(f)(x)|\less \frac{|\Gamma(\tau)|}{|\Gamma(\xi)|}|V_\tau(|f|)(x)|,
\]
which gives that 
\[
\|V_\xi(f)\|_q\less \frac{|\Gamma(\tau)|}{|\Gamma(\xi)|} \|V_\tau(|f|)\|_q.
\]
Thus \eqref{inegalite-HL} implies that 
\[
\|V_\xi(f)\|_q \less K \frac{|\Gamma(\tau)|}{|\Gamma(\xi)|} \|f\|_p,
\]
which shows that $V_\xi$ is bounded from $L^p(0,1)$ to $L^q(0,1)$ and we have $$\|V_\xi\|_{\mathcal B(L^p,L^q)}\leq K \frac{|\Gamma(\tau)|}{|\Gamma(\xi)|}\,\cdot$$

$(iv)$ Since $\xi\in\C_{\frac{1}{p}}$, then $\phi_\xi\in L^{p'}(0,1)$. Thus~$V_\xi(f)=f\ast\phi_\xi\in~C_0([0,1])$ with 
\[
\|V_\xi f\|_\infty\leq \|f\|_p \|\phi_\xi\|_{p'}.
\]
Now, using \eqref{eq:sddsssdqsdqsddnc923E0}, we get \eqref{eq:sddqdsqd02343EDSDCWl}.

$(v)$ Observe that $\phi_\xi\in C_0([0,1])$ for $\Re(\xi)\geq 1$ with $\|\phi_\xi\|_{\infty}=\dfrac{1}{|\Gamma(\xi)|}\,\cdot$ Then, argue as in $(iv)$ with $p=1$ and $p'=\infty$.

$(vi)$ Use here $p=\infty$ and $p'=1$, and argue as in $(iv)$.
\end{proof}

\begin{rem}
\rm{Let $X=C_0([0,1])$ or $X=L^p([0,1])$ for $1\less p\less \infty$. As we have just seen, the operator $V_\xi$ is bounded from $X$ into itself and
\begin{equation}\label{normVt}
\|V_\xi\|_{\mathcal B(X)}\less\dfrac{1}{\Re(\xi)|\Gamma(\xi)|}.
\end{equation}
Note that the estimate \eqref{normVt} for the case $X=C_0([0,1])$ follows immediately from Proposition~\ref{bound} and the fact that $C_0([0,1])$ is contained in $L^\infty(0,1)$. Observe also that for $X=L^p(0,1)$, the estimate \eqref{normVt}  can also be obtained easily by interpolation, or using the Minkowski inequality.
Finally, it turns out that when $X=L^1(0,1)$ and $\xi$ is real, then \eqref{normVt} becomes an equality. 

Note that in \cite{AG}, J. Adell and E. Gallardo-Guti\'errez got some lower and upper bounds for the norm of the Riemann-Liouville operator $V_t$ on $L^p(0,1)$ when $t>0$. These estimates enable them to get some asymptotic estimates of the norm. More precisely, they proved that, for every $1\less p\less \infty$, 
\[
\lim_{t\to \infty}\Gamma(t+1)\|V_t\|_{\mathcal B(L^p)}=c_{p,q}
\]
where 
\[
c_{p,q}=\begin{cases}
p^{-1/p}q^{-1/q}&\mbox{ if }1<p<\infty\\
1 &\mbox{ if $p=1$ or $p=\infty$}.  
\end{cases}
\]
However, it should be noted that the exact computation of $\|V_\xi\|_{\mathcal B(L^p)}$ is not known.}
\end{rem}
\bigskip

We now recall the notion of analytic strongly continuous semigroup. For $\theta\in (0,\pi/2]$, let $\Sigma_\theta=\{z\in\C\setminus\{0\}:|\arg(z)|<\theta\}$. If $X$ is a Banach space, the family $(T_z)_{z\in\Sigma_\theta}$ is called an \emph{analytic strongly continuous semigroup} on $X$ if it satisfies the following properties:
\begin{enumerate}[(i)]
\item $T_z\in\mathcal B(X)$;
\item for all $z,w\in\Sigma_\theta$, we have $T_zT_w=T_{z+w}$;
\item For every $\theta'\in (0,\theta)$, we have
\[
\lim_{\substack{z\to 0\\ z\in\Sigma_{\theta'}}}\|T_zf-f\|_X=0,\quad \mbox{for all }f\in X;
\]
\item  the map $z\in\Sigma_\theta\longmapsto T_z\in\mathcal B(X)$ is holomorphic.
\end{enumerate}
Note that $\C_0=\Sigma_{\pi/2}$.
\begin{prop}\label{prop:semigroup-RL}
Let $\xi\in\C_0$, let $X$ be $C_0([0,1])$ or $L^p(0,1)$ for $p\gtr1$, and let 
$$\dis V_\xi(f)(x)=\dfrac{1}{\Gamma(\xi)}\int_0^xf(u)(x-u)^{\xi-1}\;du\,,\quad\hbox{where }\,f\in X\,\hbox{ and }x\in[0,1].$$
%
Then $(V_\xi)_{\xi\in\C_0}$ is an analytic strongly continuous semigroup on $X$.  Moreover we have $V=V_1$.
\end{prop}
%
\begin{proof}
First, let us point out that, according to Proposition \ref{bound}, $V_\xi\in\mathcal B(X)$ for every $\xi\in\C_0$. 
\medskip

Now, we justify the algebraic properties of the semigroup. For $\xi$,$\xi'\in\C_0$, $f\in X$ and $x\in (0,1]$, we have 
{factVC}
$$
\begin{array}{rcl}
V_\xi V_{\xi'}(f)(x)&=&\dis\dfrac{1}{\Gamma(\xi)\Gamma(\xi')}\int_0^x (x-u)^{\xi-1}\left(\int_0^u f(s)(u-s)^{\xi'-1}\,ds\right)du\\ \medskip

&=&\dis\dfrac{1}{\Gamma(\xi)\Gamma(\xi')}\int_0^xf(s)\left(\int_s^x (u-s)^{\xi'-1}(x-u)^{\xi-1}\,du\right)ds,
\end{array}$$
using the Fubini theorem.

But, with the change of variable $u=(1-\lambda)s+\lambda x$
\begin{align*}
\int_s^x (u-s)^{\xi'-1}(x-u)^{\xi-1}\;du=&(x-s)^{\xi+\xi'-1}\int_0^1 \lambda^{\xi'-1}(1-\lambda)^{\xi-1}\;d\lambda\\
=&(x-s)^{\xi+\xi'-1}B(\xi,\xi'),
\end{align*}
where $B$ is the Beta function linked with the Gamma function by the formulae $B(\xi,\xi')=\frac{\Gamma(\xi)\Gamma(\xi')}{\Gamma(\xi+\xi')}$, for $\xi,\xi'\in\mathbb C_0$. 
We conclude that 
$$V_\xi V_{\xi'}(f)(x)=\dis\dfrac{1}{\Gamma(\xi+\xi')}\int_0^x (x-s)^{\xi+\xi'-1}f(s)\;ds=V_{\xi+\xi'}(f)(x)$$ 
and $(RL)$ has the semigroup property.
\medskip

Let us now justify that for every $\theta'\in (0,\pi/2)$ and every $f\in X$, we have
\begin{equation}\label{equation-strongly-continuous}
\lim_{\substack{\xi\to 0\\ \xi\in\Sigma_{\theta'}}}\|V_\xi f-f\|_X=0,
\end{equation}
Thanks to \eqref{normVt}, we have 
\[
\|V_\xi\|_{\mathcal B(X)}\less \frac{1}{\Re(\xi)|\Gamma(\xi)|}=\frac{|\xi|}{\Re(\xi)|\Gamma(\xi+1)|}=\frac{1}{\cos(\arg(\xi))|\Gamma(\xi+1)|}\cdot
\]
But for any $\xi\in\Sigma_{\theta'}$, we have $\cos(\arg(\xi))\gtr \cos(\theta')>0$, which proves that 
\[
\sup_{\xi\in\Sigma_{\theta'}\cap D(0,1/2)}\|V_\xi\|_{\mathcal B(X)}<\infty,
\]
where $D(0,1/2)=\{\xi\in\C:|\xi|<1/2\}$. 
Denote by $f_n(t)=t^n$, $n\gtr 1$, $t\in [0,1]$ and observe that the family $\{f_n:n\gtr 1\}$ spans a dense subspace of $X$. 
Therefore, it suffices to prove \eqref{equation-strongly-continuous} for every $f=f_n$, $n\gtr 1$, to obtain that \eqref{equation-strongly-continuous} holds for every $f\in X$. According to \eqref{RL2}, for every $n\gtr 1$, we have 
\begin{equation}\label{eq:action-monome}
V_\xi(f_n)(x)=\frac{B(n+1,\xi)}{\Gamma(\xi)}x^{n+\xi}=\frac{\Gamma(n+1)}{\Gamma(n+1+\xi)}x^{n+\xi}.
\end{equation}
Now write
\begin{eqnarray*}
\|V_\xi(f_n)-f_n\|_X&=&\left\|\frac{\Gamma(n+1)}{\Gamma(n+1+\xi)}x^{n+\xi}-x^n\right\|_X \\
&\less & \left| \frac{\Gamma(n+1)}{\Gamma(n+1+\xi)}-1\right| \|x^{n+\xi}\|_X+\|x^{n+\xi}-x^n\|_X.
\end{eqnarray*}
Observe that $\|x^{n+\xi}\|_X\less \|x^{n+\xi}\|_\infty\less 1$ and by continuity of the Gamma function, we have 
\[
\lim_{\xi\to 0}\left| \frac{\Gamma(n+1)}{\Gamma(n+1+\xi)}-1\right| \|x^{n+\xi}\|_X=0.
\]
Hence it remains to justify that for every $n\gtr 1$, we have
\[
\lim_{\xi\to 0}\|x^{n+\xi}-x^n\|_X=0.
\] 
For $X=L^p(0,1)$, this follows easily from Lebesgue's dominated convergence theorem, but can also be deduced from the following case. For $X=C_0([0,1])$, observe that for every complex numbers $a,b$, $\Re(a)\neq \Re(b)$, we have 
\begin{equation}\label{inegalite-trivial-exponentielle}
\big|\e^a-\e^b\big|\less \frac{|a-b|}{|\Re(a)-\Re(b)|}\big|\e^{\Re(a)}-\e^{\Re(b)}\big|.
\end{equation}
Indeed it is sufficient to write that 
\[
\e^a-\e^b=\int_0^1 \e^{b+t(a-b)}(a-b)\,dt,
\]
which gives that 
\[
|\e^a-\e^b|\less |a-b|\int_0^1 \e^{\Re(b)+t(\Re(a)-\Re(b))}\,dt
\]
and we now obtain easily \eqref{inegalite-trivial-exponentielle}. 
Now the inequality \eqref{inegalite-trivial-exponentielle} implies 
\[
|x^{n+\xi}-x^n|\less \frac{|\xi|}{\Re(\xi)}(x^n-x^{n+\Re(\xi)}).
\]
Denote by $\tau=\Re(\xi)$. We easily check that the function $x\longmapsto x^n-x^{n+\tau}$ reached its maximum on $(0,1)$ at point $x_{n}:=(\frac{n}{n+\tau})^{1/\tau}$. Hence we get that 
\[
\|x^{n+\xi}-x^n\|_\infty\less \frac{|\xi|}{\tau}(1-x_n^\tau)=\frac{|\xi|}{\tau}\left(1-\frac{n}{n+\tau}\right)=\frac{|\xi|}{n+\tau},
\]
which obviously implies that $\lim_{\xi\to 0}\|x^{n+\xi}-x^n\|_\infty=0$. That concludes the fact that $V_\xi$ satisfies \eqref{equation-strongly-continuous}.

It remains to prove the analyticity of $\xi\in\C_0\longmapsto V_\xi\in\mathcal B(X)$. Let us first start with $X=L^p(0,1)$, $1\less p<\infty$. It is sufficient (see \cite[Theorem 3.12]{Kato}) to show that for any $f\in L^p(0,1)$ and any $g\in L^{p'}(0,1)$, the map
\[
\xi\longmapsto \int_0^1\int_0^x g(x)(x-u)^{\xi-1}f(u)\,du\,dx
\]
is analytic on $\C_0$. But $\xi\longmapsto g(x)(x-u)^{\xi-1}f(u)$ is analytic on $\C_0$ for almost all $(x,u)\in (0,1)\times (0,1)$. Moreover, for every $\xi$ such that $\Re(\xi)>\delta>0$, we have 
\[
|g(x)(x-u)^{\xi-1}f(u)|\less |g(x)|(x-u)^{\delta-1}|f(u)|,\qquad 0\less x\less 1,\,0\less u<x
\]
and  
\[
\int_0^1 \int_0^x |g(x)|(x-u)^{\delta-1}|f(u)|\,du\,dx=\Gamma(\delta)\int_0^1 |g(x)|V_\delta(|f|)(x)\,dx<\infty
\]
since $g\in L^{p'}(0,1)$ and $V_\delta(|f|)\in L^p(0,1)$ according to Proposition~\ref{bound}. It follows from the Lebesgue's dominated convergence theorem that the map $\xi\longmapsto \int_0^1\int_0^x g(x)(x-u)^{\xi-1}f(u)\,du\,dx$ is analytic on $\C_0$. A similar argument works also for $X=C_0([0,1])$ but we have to replace $g(x)dx$, $g\in L^{p'}(0,1)$, by an arbitrary Borel measure $\mu$ on $(0,1)$. 
\end{proof}

\begin{rem}
\rm{
Writing $\xi=\tau+it$ ($\tau,t$ real with $\tau>0$), we may consider the behavior of $V_\xi$ as $\tau=\Re(\xi)\to 0^+$. For $1<p<\infty$, it was proved by Kalisch \cite{Kalisch} and Fisher \cite{Fisher} that, for each fixed $f\in L^p(0,1)$, the limit 
\[
V_{it}f=\lim_{\tau\to 0^+}V_{\tau+it}f
\]
exists in the $L^p(0,1)$ norm. Furthermore, the family of operators $(V_{it})_{t\in\mathbb R}$ so defined forms a strongly continuous group of bounded operators on $L^p(0,1)$. We also have that $V_{\xi}V_{it}=V_{it}V_{\xi}=V_{\xi+it}$ for every $\xi\in\C_0$ and $t\in\mathbb R$.}
\end{rem}

\begin{rem}
\rm{
Let $1\less p<\infty$. It is not difficult to check that for every $\xi\in\mathbb C_0$, the adjoint of $V_\xi:L^p(0,1)\to L^p(0,1)$ is the operator given by
\begin{equation}\label{eq:formule-adjoint32A23}
V_\xi^\ast(f)(x)=\dfrac{1}{\overline{\Gamma(\xi)}}\int_x^1 f(u)(u-x)^{\overline{\xi}-1}\;du,\qquad f\in L^{p'}(0,1). 
\end{equation}
Actually, we work here with a kernel operator, that is why the symmetric kernel naturally appears.}
\end{rem}

The spectrum and point spectrum are well-known and easy to identify. We also give a proof for self-completeness. 
\begin{prop}
Let $\xi\in\C_0$, and let $X$ be $C_0([0,1])$ or $L^p(0,1)$ for $p\gtr1$. Then $\sigma(V_\xi)=\{0\}$ and $V_\xi$ has no eigenvalue. 
\end{prop}

\begin{proof}
We first show that the spectral radius of $V_\xi$ is equal to $0$.
Indeed for every integer $n\gtr1$, $V_\xi^n=V_{n\xi}\;$ so, thanks to \eqref{normVt}, we have
$$\big\|V_\xi^n\big\|_{\mathcal B(X)}^{\frac{1}{n}}\less(n\Re(\xi)|\Gamma(n\xi)|)^{-\frac{1}{n}}=O\big(n^{-\Re(\xi)}\big)\longrightarrow0,\quad \mbox{as }n\to\infty,$$
thanks to the Stirling Formula. Therefore we get $\sigma(V_\xi)=\{0\}$. In particular, the point spectrum of $V_\xi$ is included in $\{0\}$. 

Now we show that $V_\xi$ is one-to-one. Indeed assume that $f\in L^1(0,1)$ satisfies $V_\xi(f)=0$. It means that for every $x\in[0,1]$, we have
$$\int_0^x f(u)(x-u)^{\xi-1}\;du=0.$$
Observe that the function $g$ defined by $g(u)=u^{\xi-1}$, $u\in (0,1)$, does not vanish on $(0,1)$. Then it follows from the Titchmarsh theorem (Theorem~\ref{thm:titchmarsh}) that $f(u)=0$ for almost all $u\in (0,1)$. We finally conclude that $V_\xi$ has no eigenvalue.
\end{proof}

\section{Unicellularity}\label{sec:unicellularity}

It is a well-known fact \cite[p. 397]{GK67} that for each $t>0$, the lattice of invariant subspaces and even hyperinvariant subspaces of $V_t$ acting on $L^2(0,1)$ is totally ordered (i.e. $V_t$ is unicellular) and described by  \[\{E_{a,2}:=\ind_{[a,1]}L^2(0,1),0<a<1\} .\] 
We recall that a closed subspace $E$ of $X$ is {\emph{hyperinvariant}} for an operator $T\in\mathcal B(X)$ if $E$ is invariant with respect to any operator commuting with $T$. Let us also recall that a vector $x\in X$ is said to be cyclic for $T$ if $\text{span}(T^nx:n\geq 0)$ is dense in $X$. The description of the invariant subspaces of $V_t$ immediately yields the following description of the cyclic functions of $V_t$, namely, 
\[ f\mbox{ is cyclic for }V_t \in{\mathcal B}(L^2(0,1))\Longleftrightarrow  \forall\varepsilon>0, \int_0^\varepsilon |f(x)|^2 dx>0. \]   
In this section, we will revisit these results for $V_\xi$ extending both to the $L^p(0,1)$ case, $1\less p<\infty$, and also to the complex case $\xi\in\mathbb C_0$. 

We first characterize the cyclic vectors for $V_\xi$ on $L^p(0,1)$, when $\xi\in \mathbb C_0$ and $1\less p<\infty$. 
In order to state the result, we introduce a convenient notation. We associate with every function $f\in L^p(0,1)$ a number $\ell_f\in [0,1]$ defined by 
\[
\ell_f=\sup\left\{\ell\in [0,1]:\int_0^\ell |f(u)|^p\,du=0\right\}.
\]
Equivalently, 
\[
f=0 \quad\text{ a.e. on }(0,\ell_f)\quad\mbox{and}\quad \forall\varepsilon>0,\quad \int_{\ell_f}^{\ell_f+\varepsilon}|f(u)|^p\,du>0.
\]
\begin{theorem}\label{thm:cyclic}
Let $1\less p<\infty$, $\xi\in\mathbb C_0$ and let $f\in L^p(0,1)$. Then $f$ is a cyclic vector for $V_\xi$ if and only if $\ell_f=0$, that is, for every $\varepsilon>0$, 
\[
\int_{0}^{\varepsilon}|f(u)|^p\,du>0.
\]
\end{theorem}

\begin{proof}
Assume first that $\ell_f>0$. Then, $f=0$  a.e. on $[0,\ell_f]$, which immediately implies that $V_\xi^nf=V_{n\xi}f=0$ on $[0,\ell_f]$. In particular, we have that $\mbox{{\rm span}}(V_\xi^n f:n\gtr 0)\subset \ind_{[\ell_f,1]}L^p(0,1)$, and $f$ is not a cyclic vector for $V_\xi$. 

Conversely, assume that $\ell_f=0$. We shall show that any $\phi\in (L^p(0,1))^*$ which satisfies $\phi(V_\xi^nf)=0$, $n\gtr 0$, is equal to zero. By the Riesz representation theorem, there exists $g\in L^{p'}(0,1)$ such that 
\[
\phi(h)=\int_0^1 h(x)g(x)\,dx, \qquad h\in L^p(0,1).
\]
Hence, since $V_\xi^n=V_{n\xi}$, for every $n\gtr 1$, we have 
\[
\int_0^1 g(x)\left(\int_0^xf(u)(x-u)^{n\xi-1}\,du\right)\,dx=0.
\]
Now, Corollary~\ref{Cor:titch} implies that there exists $\alpha\in [0,1]$ such that $f(u)=0$ for almost all $u\in [0,\alpha]$ and $g(u)=0$ for almost all $u\in [\alpha,1]$. But, since $\ell_f=0$, necessarily $\alpha=0$ and $g(u)=0$ for almost all $u\in [0,1]$. Hence $\phi=0$. By the Hahn--Banach theorem, the subspace $\mbox{{\rm span}}(V_\xi^n f:n\gtr 0)$ is dense in $L^p(0,1)$, which implies that $f$ is cyclic for $V_\xi$ in $L^p(0,1)$. 
\end{proof}

We will deduce from Theorem~\ref{thm:cyclic} the description of the invariant subspaces of $V_\xi$.

\begin{theorem}
Let $E_{\ell,p}=\ind_{[\ell,1]}L^p(0,1)$, for $0\less \ell\less 1$, $1\less p<\infty$ and let $\xi\in\mathbb C_0$. 
\begin{enumerate}[(i)]
\item The closed invariant subspaces of $V_\xi$ on $L^p(0,1)$ are exactly the subspaces $E_{\ell,p}$, $0\less \ell\less 1$. In particular, the operator $V_\xi$ is unicellular.
\item The closed hyperinvariant subspaces of $V_\xi$ on $L^p(0,1)$ are exactly the subspaces $E_{\ell,p}$, $0\less \ell\less 1$.
\end{enumerate}
\end{theorem}

\begin{proof}
$(i)$ We follow the ideas contained in the proof of Theorem~4.6 in \cite{ShapiVol}.

First, it is easily seen that each of the subspaces $E_{\ell,p}$, $0\less \ell\less 1$, is closed and invariant with respect to $V_\xi$ on $L^p(0,1)$. Conversely, let $E$ be a proper closed invariant subspace of $V_\xi$ on $L^p(0,1)$. Let $a=\inf\{\ell_f:f\in E\}$. Since $E\neq\{0\}$, we have $0\leq a<1$, and clearly $E\subset E_{a,p}$. Let us prove that we have indeed equality. 

For every $\alpha\in (a,1]$, there exists $f\in E$ such that $a\leq \ell_f<\alpha\leq 1$. Denote by $\lambda=\ell_f$. Then $f\in E_{\lambda,p}$. We introduce the following $C^1$-diffeomorphism 
\[
\begin{array}{cccl}
\varphi:&[0,1]&\longrightarrow&[\lambda,1]\\
&t&\longmapsto&(1-t)\lambda+t=(1-\lambda)t+\lambda,
\end{array}
\]
and the space 
\[
\widetilde{E}=\{g\circ \varphi:g\in E\cap E_{\lambda,p}\}.
\]
Observe that for every $g\in E_{\lambda,p}$, we have 
\[
\|g\circ\varphi\|_p^p=\frac{1}{1-\lambda}\int_{\lambda}^1 |g(\theta)|^p\,d\theta=\frac{1}{1-\lambda}\int_0^1 |g(\theta)|^p\,d\theta=\frac{1}{1-\lambda}\|g\|_p^p.
\]
In particular, we deduce that $\widetilde{E}$ is a closed subspace of $L^p(0,1)$. We also denote by $F=f\circ \varphi$. Since $f\in E\cap E_{\lambda,p}$, we get that $F\in\widetilde{E}$. Moreover, $V_\xi(\widetilde{E})\subset \widetilde{E}$. Indeed, let $g\in E\cap E_{\lambda,p}$. Then, for $x\in [0,1]$, we have
\[
\begin{aligned}
V_\xi(g\circ\varphi)(x)=&\frac{1}{\Gamma(\xi)}\int_0^x (g\circ\varphi)(u)(x-u)^{\xi-1}\,du\\
=&\frac{1}{\Gamma(\xi)}\int_\lambda^{\varphi(x)}g(\theta)\left(x-\frac{\theta-\lambda}{1-\lambda}\right)^{\xi-1}\frac{1}{1-\lambda}\,d\theta\\
=&\frac{(1-\lambda)^{-\xi}}{\Gamma(\xi)}\int_\lambda^{\varphi(x)}g(\theta)(\varphi(x)-\theta)^{\xi-1}\,d\theta,
\end{aligned}
\]
where we make the change of variable $\theta=\varphi(u)$. Since $g\in E_{\lambda,p}$, we have $g=0$ a.e. on $(0,\lambda)$, whence
\[
V_\xi(g\circ\varphi)(x)=\frac{(1-\lambda)^{-\xi}}{\Gamma(\xi)}\int_0^{\varphi(x)}g(\theta)(\varphi(x)-\theta)^{\xi-1}\,d\theta=(1-\lambda)^{-\xi}(V_\xi(g))(\varphi(x)).
\]
In other words, $V_\xi(g\circ\varphi)=(1-\lambda)^{-\xi}(V_\xi(g))\circ\varphi$. Taking account of the fact that $E\cap E_{\lambda,p}$ is invariant with respect to $V_\xi$, we deduce that $V_\xi(g)\circ \varphi\in \widetilde{E}$, and thus $V_\xi(\widetilde{E})\subset \widetilde{E}$.

Observe now that $\ell_F=0$. Indeed, for every $\varepsilon\in (0,1)$, we have 
\[
\int_0^{\varepsilon}|F(u)|^p\,du=\int_0^{\varepsilon}|f(\varphi(u))|^p\,du=\frac{1}{1-\lambda}\int_\lambda^{\varphi(\varepsilon)}|f(\theta)|^p\,d\theta,
\]
and $\varphi(\varepsilon)=(1-\lambda)\varepsilon+\lambda>\lambda=\ell_f$. Thus, by definition of $\ell_f$, we get 
\[
\int_0^\varepsilon |F(u)|^p\,du>0\qquad\text{for every }\varepsilon\in (0,1),
\]
which gives $\ell_F=0$. Therefore, according to Theorem~\ref{thm:cyclic}, $F$ is cyclic for $V_\xi$, meaning that $\text{span}(V_\xi^n F:n\geq 0)$ is dense in $L^p(0,1)$. But, since $F\in\widetilde{E}$ and $\widetilde{E}$ is invariant with respect to $V_\xi$, observe that we have $\text{span}(V_\xi^n F:n\geq 0)\subset \widetilde{E}$. Thus we deduce that $\widetilde{E}$ is dense in $L^p(0,1)$. But we have already observed that $\widetilde{E}$ is also closed, and finally we can conclude that $\widetilde{E}=L^p(0,1)$. 

Take now $h\in E_{\lambda,p}$. Then $h\circ\varphi\in L^p(0,1)=\widetilde{E}$, and there exists $g\in E\cap E_{\lambda,p}$ such that $h\circ\varphi=g\circ\varphi$. Since $h$ and $g$ vanish a.e. on $(0,\lambda)$, we deduce that $h=g\in E\cap E_{\lambda,p}$. Hence $E_{\lambda,p}\cap E=E_{\lambda,p}$, that is,  $E_{\alpha,p}\subset E_{\lambda,p}\subset E$. Thus 
\[
\bigcup_{\alpha\in (a,1]}E_{\alpha,p}\subset E\subset E_{a,p}.
\]
Since $\bigcup_{\alpha\in (a,1]}E_{\alpha,p}$ is dense in $E_{a,p}$, we get that $E=E_{a,p}$ as required. 


$(ii)$ The only thing to prove is that $E_{\ell,p}$, $0<\ell<1$, is a hyperinvariant subspace for $V_\xi$. Let $A\in\mathcal B(L^p(0,1))$ such that $AV_\xi=V_\xi A$. 
We shall prove that $A(E_{\ell,p})\subset E_{\ell,p}$. Let $\lambda\in\mathbb C$ such that $|\lambda|>\|A\|$. 
Observe that $(\lambda I-A)(E_{\ell,p})$ is an invariant subspace for $V_\xi$ and it is also closed because $\lambda I-A$ is invertible. 
Hence by unicellularity of $V_\xi$, we have either $(\lambda I-A)(E_{\ell,p})\subset E_{\ell,p}$ or $E_{\ell,p}\subset (\lambda I-A)(E_{\ell,p})$. 
In the first case, we immediately get that $A(E_{\ell,p})\subset E_{\ell,p}$, and we are done. 
So we may assume that for every $\lambda\in\mathbb C$ with $|\lambda|>\|A\|$, we have $E_{\ell,p}\subset (\lambda I-A)(E_{\ell,p})$. 
Hence $(\lambda I-A)^{-1}(E_{\ell,p})\subset E_{\ell,p}$.
 Fix $f\in E_{\ell,p}$. 
 For every $\lambda\in\mathbb C$ with $|\lambda|>\|A\|$, define $f_\lambda=(\lambda I-A)^{-1}f$. 
Then $f_\lambda \in E_{\ell,p}$, and we can rewrite the last relation as 
\[
\lambda^{-1}f+\lambda^{-2}Af+\sum_{n=2}^\infty \lambda^{-n-1}A^nf=f_\lambda.
\] 
Multiplying by $\lambda^2$ gives 
\[
Af+\lambda f-\lambda^2f_\lambda=-\sum_{n=2}^\infty \lambda^{-n+1}A^nf.
\]
Hence
\[
\begin{aligned}
\|Af+\lambda f-\lambda^2f_\lambda\|_p&\less \sum_{n=2}^\infty |\lambda|^{-n+1}\|A\|^n\|f\|_p\\
&= \frac{|\lambda|^{-1}\|A\|^2}{1-|\lambda|^{-1}\|A\|} \|f\|_p.
\end{aligned}
\]
In particular, when $|\lambda|\to\infty$, we get that $-\lambda f+\lambda^2f_\lambda\to Af$ in $L^p(0,1)$. Since $-\lambda f+\lambda^2f_\lambda$ belongs to $E_{\ell,p}$ and $E_{\ell,p}$ is closed, we deduce that $Af\in E_{\ell,p}$. Hence $A(E_{\ell,p})\subset E_{\ell,p}$, which concludes the proof. 
\end{proof}

\section{ Membership to ideals of operators}

We introduce the following notation for the duality $L^p-L^{p'}$ when $1\less p<\infty$. 
For $f\in L^p(0,1)$ and $g\in L^{p'}(0,1)$, 
\[
\langle f,g\rangle_{L^p,L^{p'}}=\int_0^1 f(x)\overline{g(x)}\,dx.
\]
For $n\in\mathbb Z$, we also denote by $e_n(x)=\e^{2i\pi nx}$, $x\in (0,1)$. Note also that, as usual, we identify $L^1(0,1)$ and $L^1(\mathbb T)$, and for $f\in L^1(0,1)$, we have
\[
\widehat{f}(n)=\langle f,e_n\rangle_{L^1,L^\infty}=\int_{0}^1 f(x)\e^{-2i\pi nx}\,dx, \qquad n\in\mathbb Z.
\]
The following technical lemma will be the key point to study the membership of the Riemann-Liouville semigroup to various classes of ideals of operators. 

\begin{lemma}\label{lem:schatten}
Let $1\less p<\infty$, and  let $\xi$ such that $0<\Re(\xi)\leq 1$. Then
\begin{equation}\label{eq:equivalent-cle}
\langle V_\xi e_n,e_n\rangle_{L^p,L^{p'}} \sim\frac{1}{(2i \pi n)^\xi},\qquad\mbox{as }n\to\infty.
\end{equation}
\end{lemma}

\begin{proof}
Using the change of variable $s=x-\theta$ in the inner integral and the Fubini theorem, we have
\begin{eqnarray}\label{eq:sdsqdsqsdq9999}
\Gamma(\xi)\langle V_\xi(e_n),e_{n}\rangle_{L^p,L^{p'}}&=&\int_0^1\e^{-2i\pi nx}\left(\int_0^x\e^{2i\pi n\theta}(x-\theta)^{\xi-1}\,d\theta\right)dx\notag\\
&=&\int_0^1(1-s)\e^{-2i\pi ns}s^{\xi-1}\;ds\notag\\
&=&\widehat{\psi_{\xi-1}}(n)-\widehat{\psi_{\xi}}(n),
\end{eqnarray}
where $\psi_\alpha(s)=s^\alpha$, $s\in (0,1)$, with $\Re(\alpha)\in(-1,1]$. Observe that, for every $n\gtr1$, we have,
\begin{equation}\label{eq:3SDSDSDSD2}
n^{\alpha+1}\widehat{\psi_{\alpha}}(n)=\int_0^nu^\alpha\e^{-2i\pi u}\;du.
\end{equation}
We apply now the Cauchy Formula to the function $$z\mapsto z^\alpha\exp(-2i\pi z)=\exp(\alpha\log(z)-2i\pi z),$$ where $\log$ is associated to an argument  taking values in $(-\pi,\pi)$, over the boundary of the sector-shaped region in the complex plane formed by $[0,-in]$ along the negative imaginary axis, a circular arc of radius $n$ centered at the origin, and $[n,0]$ along the positive real axis. 
Actually we should be careful with the origin but with an obvious limit argument in $0$, we get 
\begin{equation}\label{ssdsdsd12123242}
\int_0^nu^\alpha\e^{-2i\pi u}\;du=-i\e^{-i\frac{\pi\alpha}{2}}\int_0^nu^\alpha\e^{-2\pi u}\;du+I_n,
\end{equation}
where 
\[
I_n=\int_{-\pi/2}^0 n^{\alpha}\e^{i\alpha t}\e^{-2i\pi n\e^{it}}in\e^{it}\,dt=i n^{\alpha+1}\int_{-\pi/2}^0 \e^{i\alpha t-2i\pi n \e^{it}+it}\,dt.
\]
Let us show that
\begin{equation}\label{ssdsdsd12123242fkdfdf}
I_n\sim \frac{in^\alpha}{2\pi},\qquad\mbox{as }n\to\infty.
\end{equation}
By the change of variable $x=-nt$, we have
\[
I_n=in^{\alpha}\int_0^{n\frac{\pi}{2}}\e^{-i\alpha\frac{x}{n}} \e^{-2i\pi n\cos(\frac{x}{n})}\e^{-2\pi n \sin(\frac{x}{n})}\e^{-i\frac{x}{n}}\,dx.
\]
Using that $\e^{2i\pi n}=1$, we can write that $\e^{-2i\pi n\cos(\frac{x}{n})}=\e^{4i\pi n\sin^2(\frac{x}{2n})}$, which implies that 
\[
\frac{1}{i n^\alpha}I_n= \int_0^{\infty} f_n(x),dx.
\] 
with $f_n(x)=\e^{-i\alpha\frac{x}{n}}\e^{4i\pi n\sin^2(\frac{x}{2n})}\e^{-2\pi n \sin(\frac{x}{n})}\e^{-i\frac{x}{n}}\ind_{(0,n\frac{\pi}{2})}(x)$, $x>0$. Observe that for every $x>0$, we have 
\[
f_n(x)\to \e^{-2\pi x}\qquad \mbox{as }n\to \infty.
\]
Moreover, we have
\[
|f_n(x)|=\e^{-2\pi n\sin(\frac{x}{n})}\e^{\Im(\alpha)\frac{x}{n}}\ind_{(0,n\frac{\pi}{2})}(x).
\]
But for $x\in (0,n\frac{\pi}{2})$, we have
\[
\e^{\Im(\alpha)\frac{x}{n}} \less \max(1,\e^{\Im(\alpha)\frac{\pi}{2}}),
\]
and since $\sin(\frac{x}{n})\gtr \frac{2}{\pi}\frac{x}{n}$, we deduce that 
\[
|f_n(x)|\less f(x):=\max(1,\e^{\Im(\alpha)\frac{\pi}{2}}) \e^{-4x},
\]
and $f$ belongs to $L^1((0,\infty))$. Therefore, we can apply the dominated Lebesgue convergence theorem, and we get 
that 
\[
\frac{1}{in^\alpha}I_n\to \int_0^\infty \e^{-2\pi x}\,dx=\frac{1}{2\pi},\qquad \mbox{as }n\to\infty,
\]
which finally gives \eqref{ssdsdsd12123242fkdfdf}.

According to \eqref{eq:3SDSDSDSD2}, \eqref{ssdsdsd12123242} and \eqref{ssdsdsd12123242fkdfdf}, we can write
\begin{align*}
n^{\alpha+1}\widehat{\psi_{\alpha}}(n)&=-i\e^{-i\frac{\pi\alpha}{2}}\int_0^\infty u^\alpha\e^{-2\pi u}\;du+o(1)+\frac{in^\alpha}{2\pi}+o(n^{\alpha})\\
&=\frac{\Gamma(\alpha+1)}{(2i\pi)^{\alpha+1}}+o(1)+\frac{in^\alpha}{2\pi}+o(n^{\alpha}),
\end{align*}
which implies that 
\begin{equation}\label{eq:coeff-Fourier-psi-alpha}
\widehat{\psi_{\alpha}}(n)= \frac{\Gamma(\alpha+1)}{(2i\pi n)^{\alpha+1}}+o\left(\frac{1}{n^{\alpha+1}}\right)+\frac{i}{2\pi n}+o\left(\frac{1}{n}\right).
\end{equation}
According to \eqref{eq:sdsqdsqsdq9999} and taking into account that $o(1/n)=o(1/n^{\xi})$ (since $0<\Re(\xi)\less 1$),  we thus have 
\begin{align*}
\Gamma(\xi)\langle V_\xi(e_n),e_{n}\rangle_{L^p,L^{p'}}&=\frac{\Gamma(\xi)}{(2i\pi n)^{\xi}}+o\left(\frac{1}{n^{\xi}}\right)+\frac{i}{2\pi n}+o\left(\frac{1}{n}\right)\\
&-\left(\frac{\Gamma(\xi+1)}{(2i\pi n)^{\xi+1}}+o\left(\frac{1}{n^{\xi+1}}\right)+\frac{i}{2\pi n}+o\left(\frac{1}{n}\right) \right)\\
&=\frac{\Gamma(\xi)}{(2i \pi n)^\xi}+o\left(\frac{1}{n^\xi}\right),
\end{align*}
which concludes the proof.
\end{proof}

\begin{rem}
\rm{Note that in the case when $\xi=1$, it is easy to check that \eqref{eq:equivalent-cle} is indeed an equality. Moreover, in the case when $0<\Re(\xi)<1$, we can also obtain \eqref{eq:equivalent-cle} more quickly. Indeed, simple estimates (using the concavity of the sinus on $(0,\pi/2)$) show that $I_n=O\left(n^{\Re(\alpha)}\right)$. This gives 
\begin{equation}\label{calculcrucial}
\widehat{\psi_{\alpha}}(n)=-i\frac{\e^{-i\frac{\pi\alpha}{2}}}{n^{1+\alpha}}\int_0^nu^\alpha\e^{-2\pi u}\;du+O\Big(\frac{1}{n}\Big)\;\cdot 
\end{equation}
Using that $\Re(\xi)<1$, we thus obtain that 
\begin{align*}
\Gamma(\xi)\langle V_\xi(e_n),e_{n}\rangle_{L^p,L^{p'}}&=\widehat{\psi_{\xi-1}}(n)-\widehat{\psi_{\xi}}(n)\\
&=\frac{\e^{-i\frac{\pi \xi}{2}}}{n^\xi}\int_0^nu^{\xi-1}\e^{-2\pi u}\;du+O\Big(\frac{1}{n}\Big)+O\Big(\frac{1}{n^{\xi+1}}\Big)\\
&=\frac{\e^{-i\frac{\pi \xi}{2}}}{(2\pi n)^\xi}\int_0^{2\pi n}u^{\xi-1}\e^{-u}\;du+o\Big(\frac{1}{n^{\xi}}\Big)\\
&\sim\frac{\Gamma(\xi)}{(2i \pi n)^\xi}.
\end{align*}}
\end{rem}

\subsection{Around compactness}

\begin{prop}\label{prop-Vxi-psumming-xisuperieur-1p}
Let $1\less p<\infty$ and let $\xi\in\C_{\frac{1}{p}}$. Then the operator $V_\xi$ is absolutely $p$-summing on $L^p(0,1)$. 
\end{prop}
\begin{proof}
According to Proposition~\ref{bound}, for $\xi\in\C_{\frac{1}{p}}$, the operator $V_\xi$ factorizes through the formal identity $i_p$ from $C([0,1])$ to $L^p(0,1)$ with the following commutative diagram

\[
\xymatrix{
    L^p(0,1) \ar[r]^{V_\xi}\ar[d]_{V_\xi} & L^p(0,1) \\
    C([0,1]) \ar[ru]_{i_p} & 
  }
  \]
According to \eqref{Diestel-eq3}, we know that $i_p$ is absolutely $p$-summing, whence we get, by ideal property, that $V_\xi$ is also absolutely $p$-summing on $L^p(0,1)$. 
\end{proof}

\begin{rem}
\rm{
We deduce from Proposition~\ref{prop-Vxi-psumming-xisuperieur-1p} that for $\xi\in\C_{\frac{1}{p}}$, the operator  $V_\xi$ is weakly compact  and completely continuous on $L^p(0,1)$ (every absolutely $p$-summing operator is weakly compact and completely continuous). 
Now since $L^p(0,1)$ is reflexive for $1<p<\infty$, we obtain that $V_\xi$ is compact on $L^p(0,1)$, $1<p<\infty$ and $\xi\in\C_{\frac{1}{p}}$.  Nevertheless, this latter property occurs more often as showed in the following result. }
\end{rem}
Note that the following result is certainly known to experts. 
However, we provide here a very short proof.

\begin{prop}\label{prop:compacite-C_0}

For every $\xi\in\C_0$, $V_\xi$ is compact from $L^p(0,1)$ into itself.
It is also compact from $C([0,1])$ to $C_0([0,1])$.
\end{prop}

\begin{proof} Since $\phi_\xi\in L^1(0,1)$ (recall \eqref{phi}), the convolution operator by $\phi_\xi$, hence $V_\xi$, is compact.

Let us provide a quick selfcontained detailed argument in the case $X=L^p(0,1)$. Let $\psi:[-1,1]\to\C$ the function equal to $\phi_\xi$ on $(0,1)$ and $0$ on $[-1,0]$.
There exists a sequence of polynomials $(Q_n)_{n\gtr1}$ on $[-1,1]$ converging to $\psi$ in the space $L^1(-1,1)$.

For every $x,u\in[0,1]$, we can write $Q_n(x-u)=\sum_{j=0}^{d_n}x^jP_{j,n}(u)$ where $d_n\gtr1$ and $P_{j,n}$ are polynomials (knowing only that they are $L^{p'}$ functions would be sufficient).

Let $K_n(f)=\dis\sum_{j=0}^{d_n}X^j\dis\int_0^1 f(u)P_{j,n}(u)\,du$. It clearly defines a finite rank operator on $L^p(0,1)$.

On the other hand, for every $f\in L^p(0,1)$, we have, for a.e. $x\in(0,1)$,\vskip-10pt

\begin{align*}
(V_\xi(f)-K_n(f))(x)&=\int_0^1\big(\psi-Q_n)(x-u)f(u)\,du\\
&=\int_{x-1}^x\big(\psi-Q_n)(s)f(x-s)\,ds\\
&=\int_{-1}^1\big(\psi-Q_n)(s)\ind_{[s,s+1]}(x)F(x-s)\,ds\\
\end{align*}
where $F=f$ on $(0,1)$ and $0$ on $[-1,0]$. We get 
\[
\|V_\xi(f)-K_n(f)\|_p\less\|\psi-Q_n\|_{L^1(-1,1)}\|f\|_p.
\] 
Thus $\|V_\xi-K_n\|_{\mathcal B(L^p)}\less \|\psi-Q_n\|_{L^1(-1,1)}$, and $V_\xi$ is the limit of a sequence of finite rank operators on $L^p(0,1)$. Thus $V_\xi$ is compact on $L^p(0,1)$.

On $C([0,1])$ we could use Ascoli's theorem, but we prefer here a direct argument based on the first preceding integral expression above. Indeed, we have $\dis\|V_\xi(f)-K_n(f)\|_\infty\less\|\psi-Q_n\|_{L^1(-1,1)}\|f\|_\infty$ for every $f\in C([0,1])$.
We already know that $V_\xi(f)(0)=0$.
The sequence of finite rank operators defined by $\tilde K_n(f)=K_n(f)-K_n(f)(0)$ takes its values in $C_0([0,1])$.
Moreover, we have $\dis\|V_\xi-\tilde K_n\|\less2\|\psi-Q_n\|_{L^1(-1,1)}$ so $V_\xi$ is compact from $C([0,1])$ to $C_0([0,1])$.
\end{proof}
Point out that we can also prove immediatly the second part of Proposition \ref{prop:compacite-C_0} just using Proposition \ref{bound}.
Indeed choose $p$ large enough so that $\Re(\xi/2)>1/p$. Then $V_{\frac{\xi}{2}}$ viewed from $L^p(0,1)$ to $C_0([0,1])$ is bounded, a fortiori $V_{\frac{\xi}{2}}$ viewed from $C([0,1])$ to $C_0([0,1])$ is $p$-summing since it factorizes through the identity from $C([0,1])$ to $L^p(0,1)$. 
Therefore $V_{\frac{\xi}{2}}$ is both weakly compact and Dunford-Pettis hence $V_\xi=V_{\frac{\xi}{2}}^2$ is compact.\\

In the same spirit, we can now improve Proposition~\ref{bound}. 
\begin{corol}\label{cor:compact-sdqsdqsdqs2}
Let $1\less p\less \infty$ and let $\xi\in\C_{\frac{1}{p}}$. Then the operator $V_{\xi}:L^p(0,1)\to C_0([0,1])$ is compact.
\end{corol}

\begin{proof}
Let $\eps>0$ such that $\Re(\xi)>\frac{1}{p}+\eps$ and let $\alpha:=\xi-\eps$. Then $\alpha\in\C_{\frac{1}{p}}$ and we have the following commutative diagram 
\[
\xymatrix{
    L^p(0,1) \ar[r]^{V_\xi}\ar[d]_{V_\eps} &C_0([0,1]) \\
    L^p(0,1) \ar[ru]_{V_\alpha} & 
  }.
  \]
  According to Proposition~\ref{prop:compacite-C_0}, $V_\eps:L^p(0,1)\to L^p(0,1)$ is compact and according to Proposition~\ref{bound}, $V_\alpha:L^p(0,1)\to C_0([0,1])$ is bounded. Thus the operator $V_{\xi}:L^p(0,1)\to C_0([0,1])$ is compact, by the ideal property of compact operators.
\end{proof}

\begin{rem}
\rm{Note that if $\Re(\xi)=1$, according to Proposition~\ref{bound}, the operator $V_\xi:L^1(0,1)\to C_0([0,1])$ is bounded. 
However, it cannot be compact, not even weakly compact. Indeed, consider $g_n(t)=(n+1)t^n$, $n\gtr 1$, $t\in (0,1)$. Then $g_n$ is in the unit ball of $L^1(0,1)$ and according to \eqref{eq:action-monome}, we have 
\[
V_\xi(g_n)(x)=\frac{\Gamma(n+2)}{\Gamma(n+1+\xi)}x^{n+\xi},\qquad x\in (0,1), n\gtr 1.
\]
Now a standard argument shows that the sequence $(V_\xi(g_n))_{n\gtr 1}$ cannot have a weakly convergent subsequence in $C_0([0,1])$, which shows that $V_\xi$ is not weakly compact, and thus not compact from $L^1(0,1)$ to $C_0([0,1])$. 

Nevertheless, it is proved in \cite{Lefevre} that $V_1$ is finitely strictly singular from $L^1(0,1)$ to $C_0([0,1])$. 
 }
\end{rem}

\subsection{Schatten classes}
In Hilbert spaces, an (often) easier way to check compactness of an operator is to prove its membership to the Hilbert--Schmidt class.
\begin{prop}\label{RLHS}
The operator $V_\xi$ is Hilbert-Schmidt on $L^2(0,1)$ if and only if $\xi\in\C_{\frac{1}{2}}$. Moreover, in this case we have  
$$\big\|V_\xi\|_{\mathcal S^2}=\frac{1}{|\Gamma(\xi)|\sqrt{2\Re(\xi)(2\Re(\xi)-1)}}\;\cdot$$
\end{prop}

\begin{proof}
Observe that $V_\xi$ is a kernel operator with kernel equals to 
\[
K_\xi(x,u)=\frac{1}{\Gamma(\xi)}\ind_{(0,x)}(u)(x-u)^{\xi-1}, \qquad (x,u)\in (0,1)^2,\xi\in\C_0.
\]
Then, it is well known  that $V_\xi$ is Hilbert-Schmidt on $L^2(0,1)$ if and only if $K_\xi\in L^2((0,1)^2)$ and in this case $\|V_\xi\|_{\mathcal S^2}=\|K_\xi\|_{L^2((0,1)^2)}$. Note that 
\begin{eqnarray*}
\|K_\xi\|^2_{L^2((0,1)^2)}&=&\int_0^1\int_0^1 \frac{1}{|\Gamma(\xi)|^2}\ind_{(0,x)}(u)|(x-u)^{\xi-1}|^2\,du\,dx\\
&=&\frac{1}{|\Gamma(\xi)|^2}\int_0^1 \int_0^x (x-u)^{2(\tau-1)}\,du\,dx,
\end{eqnarray*}
where $\tau=\Re(\xi)$. We see that the last quantity is finite if and only if $2\tau-1>0$, that is $\tau>1/2$. Therefore, $V_\xi$ is Hilbert-Schmidt on $L^2(0,1)$ if and only if $\xi\in\C_{1/2}$ and in this case, we have
\[
\big\|V_\xi\|^2_{\mathcal S^2}=\frac{1}{|\Gamma(\xi)|^2}\int_0^1 \frac{x^{2\tau-1}}{2\tau-1}\,dx=\frac{1}{2\tau(2\tau-1)|\Gamma(\xi)|^2},
\]
which gives the result.
\end{proof}

\begin{rem}
\rm{
We can also recover a part of Proposition~\ref{RLHS} as a consequence of  Proposition \ref{prop-Vxi-psumming-xisuperieur-1p}. Indeed,  when $H$ and $K$ are Hilbert spaces, according to \eqref{Diestel-eq10}, we have $\Pi_2(H,K)=\mathcal S^2(H,K)$.}
\end{rem}

We are going to extend Proposition~\ref{RLHS}  and characterize the membership of Schatten classes for the Riemann-Liouville analytic semigroup.

\begin{theorem}\label{RLSr}
Let $\xi\in\C_0$ and $r\gtr 1$. The following are equivalent:
\begin{enumerate}[(i)]
\item The operator $V_\xi:L^2(0,1)\to L^2(0,1)$ belongs to the Schatten class $\mathcal{S}^r$.
\item $\xi\in\C_{\frac{1}{r}}$.
\end{enumerate}
\end{theorem}
\begin{proof}
$(i)\implies (ii)$: we use reductio ad absurdum and assume that $V_\xi\in\mathcal{S}^r$ and $\Re(\xi)\less \frac{1}{r}$. From \eqref{cnschatten1}, we must have, for every orthonormal sequences $\big(h_n\big)_{n\gtr1}$ of $L^2(0,1)$,
\begin{equation}\label{condnecschatten}
\sum_{n\gtr1}\big|\langle V_\xi(h_n),h_n\rangle_{2}\big|^r<\infty\;.
\end{equation}
Here $\langle \cdot,\cdot\rangle_2$ stands for the scalar product on $L^2(0,1)$ (which concides with the duality bracket $\langle \cdot,\cdot \rangle_{L^2,L^2}$). We are going to test this condition on the family $\big(e_n\big)_{n\in\mathbb Z}$ with $e_n(\theta)=\e^{2i\pi n\theta}$. Since $0<\Re(\xi)\less\frac{1}{r}\less 1$, it follows from Lemma~\ref{lem:schatten} that
\[
\left| \langle V_\xi e_n,e_n\rangle_{L^p,L^{p'}}\right| \sim\frac{1}{(2\pi n)^{\Re(\xi)}},\qquad\mbox{as }n\to\infty.
\]
Thus \eqref{condnecschatten} would imply that 
\[
\sum_{n\gtr 1}\frac{1}{n^{r\Re(\xi)}}<\infty,
\]
giving a contradiction with $r\Re(\xi)\less 1$. Thus, we conclude that for $r\gtr 1$, if $V_\xi\in\mathcal{S}^r$ then $\Re(\xi)>\frac{1}{r}\,\cdot$\medskip

$(ii)\implies (i)$: Let $r\gtr 1$. We want to prove that the condition $\Re(\xi)>\frac{1}{r}$ implies that $V_\xi\in\mathcal{S}^r$. 

We fix $\xi\in\C_{\frac{1}{r}}$, we can choose $\eps>0$ such that $\tau=\Re(\xi)=\dfrac{1}{r}+\eps$ and we shall use Theorem~\ref{interpolschatten} with $\alpha_0=\eps$, $\alpha_1=1+\eps$, $p_0=\infty$ and $p_1=1$ and
\[
T_z=\Gamma(z)(\Gamma(z/2))^2\,V_z, \qquad z\in\overline{\Omega},
\]
where $\Omega=\{z\in\C: \eps<\Re(z)<1+\eps\}$.
First note that, according to Proposition~\ref{prop:semigroup-RL}, the map $z\longmapsto T_z$ is holomorphic from $\Omega$ into $\mathcal B(L^2(0,1))$. 
Moreover, using \eqref{normVt},  for every $\xi\in\C_0$, we have 
\[
\|T_\xi\|_{\mathcal B(L^2)} \less\frac{(\Gamma(\tau/2))^2}{\tau}\;, 
\]
so that $\dis\|T_\xi\|_{\mathcal B(L^2)}$ is bounded for $\xi\in\overline{\Omega}$. Observe now that when $\Re(z)=\eps$, $T_z$ is compact (according to Proposition~\ref{prop:compacite-C_0}), whence $T_z\in\mathcal S^\infty(L^2(0,1))$, and 
\[
\|T_z\|_{\mathcal S^\infty}=\|T_z\|_{\mathcal B(L^2)}\less\frac{(\Gamma(\eps/2))^2}{\eps}\,\cdot
\]

On the other hand, for $\xi\in\C_1$, $\frac{\xi}{2}\in\C_{\frac{1}{2}}$ and Proposition~\ref{RLHS} implies that $V_{\frac{\xi}{2}}\in\mathcal S^2(L^2(0,1))$ and 
\[
\|V_{\frac{\xi}{2}}\|_{\mathcal S^2}\less \frac{1}{|\Gamma(\frac{\xi}{2})|\sqrt{\Re(\xi)(\Re(\xi)-1)}}.
\] 
Then according to \eqref{eq:S2-S2=S1}, $V_\xi=V_{\frac{\xi}{2}}^2\in\mathcal S^1(L^2(0,1))$ and 
\[
\|V_\xi\|_{\mathcal S^1}\less \|V_{\frac{\xi}{2}}\|^2_{\mathcal S^2}\less \frac{1}{|\Gamma(\frac{\xi}{2})|^2\Re(\xi)(\Re(\xi)-1)}.
\]
Hence, for any $z\in\C_1$, $T_z\in\mathcal S^1(L^2(0,1))$ and for $\Re(z)=1+\eps$, we have 
\[
\|T_z\|_{\mathcal S^1}\less\frac{|\Gamma(z)||\Gamma(\frac{z}{2})|^2}{|\Gamma(\frac{z}{2})|^2\Re(z)(\Re(z)-1)}\less \frac{|\Gamma(1+\eps)|}{\eps(1+\eps)}.
\]
From Theorem \ref{interpolschatten} (with $\theta=\frac{1}{r}$), we get that, $T_\xi\in\mathcal{S}^{r}(L^2(0,1))$, whence $V_\xi\in\mathcal{S}^{r}(L^2(0,1))$.\end{proof}


\subsection{Characterization of nuclear and absolutely $p$-summing operators on $L^p(0,1)$.}

We can first recover both the absolutely $p$-summing property and the Hilbert-Schmidt property (when $p=2$) as a consequence of the following result.

\begin{prop}\label{OB}
Let $\xi\in\C_0$ and $1\less p,q<\infty$.
\begin{enumerate}[(i)]
\item The operator $V_\xi:L^p(0,1)\to L^q(0,1)$ is order bounded if and only if $\xi\in\C_{\frac{1}{p}}$.
\item The operator $V_\xi^*:L^p(0,1)\to L^q(0,1)$ is order bounded if and only if $\xi\in\C_{\frac{1}{p}}$, where $V_\xi^*$ is given by \eqref{eq:formule-adjoint32A23}.
\end{enumerate}
\end{prop}

\begin{proof}
Denote by $\tau=\Re(\xi)$. 

$(i)$ Denote by $K_\xi(x,u)=\ind_{(0,x)}(u)(x-u)^{\xi-1}$, $(x,u)\in (0,1)^2$.
We have 
\[
\sup_{\|f\|_p\less1}\Bigg|\int_0^xf(u)(x-u)^{\xi-1}\;du\Bigg|=\| K_\xi(x,\cdot)\|_{p'}
\]
where $p'$ is the conjugate exponent of $p$. Observe that the function $K_\xi(x,\cdot)$ belongs to $L^{p'}(0,1)$ if and only $\dis(\tau-1)p'+1>0$, which is equivalent to $\tau>\frac{1}{p}$. Moreover, in this case, we have
\[
\|K_\xi(x,\cdot)\|_{p'}=\begin{cases} 
\frac{x^{\tau-\frac{1}{p}}}{\big((\tau-1)p'+1\big)^{\frac{1}{p'}}}, & \mbox{if }1<p<\infty\\
x^{\tau-1},& \mbox{if }p=1.
\end{cases}
\]
It remains to observe that in both cases, the latter function lies in $L^q(0,1)$ (actually even in $C([0,1]))$. 

$(ii)$ The proof is similar and left to the reader, replacing the kernel $K_\xi$ of $V_\xi$ by the kernel  $K^*_\xi(x,u)=\ind_{(x,1)}(u)(x-u)^{\bar\xi-1}$, $(x,u)\in (0,1)^2$, of $V_\xi^*$. 
\end{proof}

\begin{corol}\label{cor:p-integral}
Let $1\less p,q<\infty$ and $\xi\in\C_{\frac{1}{p}}$. Then both operators $V_\xi$ and $V_\xi^*$ are $q$-integral and hence absolutely $q$-summing from $L^p(0,1)$ to $L^q(0,1)$.
\end{corol}
\begin{proof}
It follows from Theorem~\ref{Diestel-theorem8} $(i)$ that any order bounded operator $u:X\to L^q(\mu)$ is $q$-integral and according to \eqref{Diestel-eq7}, we have $\mathcal I_q(X,Y)\subset \Pi_q(X,Y)$. The result follows now immediately from Proposition~\ref{OB}.

Alternatively, we can factorize $V_\xi$ as 
\[
\xymatrix{
    L^p(0,1) \ar[r]^{V_\xi}\ar[d]_{V_\xi} &L^q(0,1) \\
    C([0,1]) \ar[ru]_{i_q} & 
  },
  \]
  which shows that $V_\xi$ is $q$-integral by definition.
\end{proof}
\medskip

For the membership to the class of nuclear operators, we need to recall some known facts on Fourier multipliers on the Hardy space $H^p=\{g\in L^p(0,1):\hat g(n)=0,\, n<0\}$, $p\gtr 1$. 
\begin{lemma}\label{lem:fourier-multiplier}
Let $p\gtr 1$. For every $y\in\R$, we consider the Fourier multiplier $M_y$ defined on the analytic polynomials $p$ by 
\begin{equation}\label{eq:fourier-multiplier}
M_y(p)(\theta)=\sum_{n=1}^{\infty}n^{iy}\hat{p}(n)\e^{2i\pi n\theta},\qquad \theta\in (0,1).
\end{equation}
Then $M_y$ is well-defined and bounded on $H^p$.
\end{lemma} 

\begin{proof}
The case $p>1$ follows from a result of Marcinkiewicz \cite[Theorem 1]{Mar}. Indeed if $(\lambda_n)_{n\geq 1}$ is the sequence defined by $\lambda_n=n^{iy}$, $n\gtr 1$, it is sufficient to check that $(\lambda_n)_{n\geq 1}$ is bounded and satisfies 
\begin{equation}\label{eq:Marc-CS}
\sup_{n\geq 0}\sum_{k=2^n}^{2^{n+1}}|\lambda_{k+1}-\lambda_k|<\infty.
\end{equation}
The boundedness of $(\lambda_n)_{n\geq 1}$ is clear and for the second condition, note that
\begin{eqnarray*}
\sum_{k=2^n}^{2^{n+1}}|\lambda_{k+1}-\lambda_k|&=& \sum_{k=2^n}^{2^{n+1}}\left|\sin\left(\frac{y\ln(1+\frac{1}{k})}{2}\right)\right|\\
&\less& \frac{|y|}{2}\sum_{k=2^n}^{2^{n+1}}\frac{1}{k}\\
&\less & |y|,
\end{eqnarray*}
which proves \eqref{eq:Marc-CS}. Thus $M_y$ is bounded on $H^p$ when $p>1$.

The case $p=1$ follows from a result of Daly--Fridli \cite[Theorem 2.1]{DF}. Indeed, we need to check that for $r>1$, we have
\begin{equation}\label{eq:Daly-CS}
\sup_{j\geq 1}\left(2^j\left(\sum_{k=2^j}^{2^{j+1}-1}\frac{|\lambda_{k+1}-\lambda_k|^r}{2^j}\right)^{1/r}\right)<\infty.
\end{equation}
Note that 
\begin{eqnarray*}
2^j\left(\sum_{k=2^j}^{2^{j+1}-1}\frac{|\lambda_{k+1}-\lambda_k|^r}{2^j}\right)^{1/r}&\leq & \frac{|y|}{2}2^j\left(\frac{1}{2^j}\sum_{k=2^j}^{2^{j+1}-1}\frac{1}{k^r}\right)^{1/r}\\
&\leq& \frac{|y|}{2}2^j\left(\frac{1}{2^j}\frac{1}{2^{jr}}2^j\right)^{1/r}\\
&=&\frac{|y|}{2},
\end{eqnarray*}
which proves \eqref{eq:Daly-CS}. Thus $M_y$ is bounded on $H^1$. In fact, formally, the result of Daly--Fridli tells us that $M_y$ is a multiplier for the real Hardy space $H_{2\pi}=\{g:\T\to\R: g\in L^1(\T)\, \mbox{ and }\widetilde{g}\in L^1(\T)\}$, where $\tilde{g}$ is the harmonic conjugate of $g$.  However, it is not difficult to see that if $M_y$ is a multiplier for $H_{2\pi}$, it is a multiplier for $H^1$. 
\end{proof}

We will now characterize the membership to the class of nuclear operators.

\begin{theorem}\label{RLnuclear}
Let $\xi\in\C_0$ and $p\gtr1$.
The following assertions are equivalent:
\begin{enumerate}[(i)]
\item The operator $V_\xi:L^p(0,1)\to L^p(0,1)$ is nuclear.
\item $\xi\in\C_1$.
\end{enumerate}
In particular, if $\xi\in\C_1$, we recover that the operator $V_\xi$ belongs to the Schatten class $\mathcal{S}^1$ on $L^2(0,1)$.
\end{theorem}
%

\begin{proof}
$(ii)\implies (i)$:  Let $\xi\in\C_1$. We decompose the proof into two cases, depending whether $1\less p\less 2$ or $p>2$.

Let us first start with $1\less p\less 2$. We will show that the operator $V_\xi:L^p(0,1)\to L^p(0,1)$ can be factorized as the composition of two absolutely $2$-summing operators. Indeed, observe that we have the following commutative diagram: 
\def\commutatif{\ar@{}[rrd]|{\circlearrowleft}}
\[ 
\xymatrix{
 L^p(0,1)\ar[rr]^{V_\xi} \ar[d]_i \commutatif &  & L^p(0,1)\\
L^1(0,1)\ar[r]_{V_{\frac{\xi}{2}}} &  L^2(0,1)\ar[r]_{V_{\frac{\xi}{2}}}& C([0,1])\ar[u]_{i_p}
}
\]
where $i$ and $i_p$ are the natural inclusions. Since $\frac{\xi}{2}\in\C_{\frac{1}{2}}$, Proposition~\ref{bound} implies that $V_{\frac{\xi}{2}}$ is bounded from $L^1(0,1)$ to $L^2(0,1)$ and Proposition~\ref{bound} implies that $V_{\frac{\xi}{2}}$ is bounded from $L^2(0,1)$ to $C([0,1])$. 
Now the Grothendieck theorem (Theorem~\ref{Diestel-theorem4}) implies that $V_{\frac{\xi}{2}}$ is absolutely $1$-summing from $L^1(0,1)$ to $L^2(0,1)$, and then 
\[
V_{\frac{\xi}{2}}\circ i:L^p(0,1)\to L^2(0,1)\mbox{ is absolutely $2$-summing},
\]
by ideal property and \eqref{Diestel-eq2}. On the other hand, by \eqref{Diestel-eq3}, the canonical inclusion $i_p:C([0,1])\to L^p(0,1)$ is absolutely $p$-summing. Thus 
\[
i_p\circ V_{\frac{\xi}{2}}:L^2(0,1)\to L^p(0,1)\mbox{ is absolutely $p$-summing},
\]
hence absolutely $2$-summing since $p\less 2$. But, according to Theorem~\ref{Diestel-theorem9} $(iii)$, the composition of two absolutely $2$-summing operators is nuclear. It follows  that $V_\xi=(j\circ V_{\frac{\xi}{2}})\circ (V_{\frac{\xi}{2}}\circ i)$ is nuclear on $L^p(0,1)$. 

Assume now that $p>2$. 
We can write $\xi=a+b+c$, where $a\in\C_{\frac{1}{p}}$, $b\in\C_0$ and $c\in\C_{\frac{1}{p'}}$, and $p'$ is the conjugate exponent of $p$ (indeed, since $\Re(\xi)>1$, we can choose $\varepsilon>0$ such that $\Re(\xi)>1+\varepsilon$, and then set $a=\frac{1}{p}+\frac{\varepsilon}{2}$, $c=\frac{1}{p'}+\frac{\varepsilon}{2}$ and $b=\xi-a-c$). Then we can factorize $V_\xi$ as the composition of three operators
\[
V_\xi=V_c\circ V_b\circ V_a,
\]
where each of them acts from $L^p(0,1)$ into itself. 
Since $a\in\C_{\frac{1}{p}}$, the operator $V_a$ is $p$-integral on $L^p(0,1)$ by Corollary~\ref{cor:p-integral} and Proposition~\ref{prop:compacite-C_0} implies that $V_b$ is compact. 
Since the product of a compact operator with a $p$-integral operator is $p$-nuclear (see Theorem~\ref{Diestel-theorem9} $(i)$), the operator $V_b\circ V_a$ is $p$-nuclear on $L^p(0,1)$. According to Theorem~\ref{Diestel-theorem9} $(ii)$, it remains to check that $V_c$ is absolutely $p'$-summing to conclude that $V_\xi$ is nuclear. 
For that purpose, we will decompose $V_c$, writing $c=c_1+c_2$, where $c_1\in\C_{\frac{1}{p}}$ and $c_2\in\C_{\frac{1}{p'}-\frac{1}{p}}$. 
Now, since $c_1\in\C_{\frac{1}{p}}$, the operator $V_{c_1}$ is bounded from $L^p(0,1)$ to $C([0,1])$ by Proposition~\ref{bound}. 
Moreover, since $p>2$, then $p'<2<p$ and it follows from Proposition~\ref{bound} that $V_{c_2}$ is bounded from $L^{p'}(0,1)$ to $L^p(0,1)$. 
Thus we have the following commutative diagram
\def\commutatif{\ar@{}[rd]|{\circlearrowleft}}
\[
 \xymatrix{
   L^p(0,1) \ar[r]^{V_c} \ar[d]_{V_{c_1}} \commutatif & L^p(0,1) \\
    C([0,1]) \ar[r]_{i_{p'}} & L^{p'}([0,1])\ar[u]_{V_{c_2}}
  }
\]
where $i_{p'}$ is the canonical injection. 
But, as already observed the map $i_{p'}:C([0,1])\to L^{p'}(0,1)$ is a $p'$-summing operator, whence $V_c$ is also $p'$-summing on $L^p(0,1)$ by the ideal property. Hence $V_\xi$ is nuclear, which concludes the proof of the first implication.

$(i)\implies (ii)$: According to the semigroup property and the ideal property of the class of nuclear operators, it is sufficient to check that for $\xi=1+iy$, $y\in\mathbb R$, $V_\xi$ is not nuclear on $L^p(0,1)$. 
we use reductio ad absurdum: we assume that there exist two sequences $(h_k)_{k\gtr 0}$ in $L^{p'}(0,1)$ and $(g_k)_{k\gtr 0}$ in $L^p(0,1)$ such that 
\[
V_\xi(f)=\sum_{k=0}^\infty \langle f,h_k\rangle g_k,\qquad f\in L^p(0,1),
\]
and 
\begin{equation}\label{eq:nuclear-absurd}
\sum_{k=0}^{+\infty}\|h_k\|_{p'}\,\|g_k\|_p<\infty
\end{equation}
For simplicity,  for $f\in L^p(0,1)$ and $g\in L^{p'}(0,1)$, $\langle f,g \rangle$ stands for $\langle f,g\rangle_{L^p,L^{p'}}$. 

For every integer $n\gtr 1$, we then have 
\[
\langle V_\xi(e_n),e_{n}\rangle=\sum_{k=0}^{+\infty} \langle e_n,h_k\rangle \langle g_k,e_{n}\rangle.
\]
where $e_n(\theta)=\e^{2i\pi n\theta}$. Hence, for any integer $N\gtr 1$, we get
\begin{eqnarray*}
\sum_{n=1}^N n^{iy} \langle V_\xi(e_n),e_{n}\rangle&=&\sum_{k=0}^\infty \sum_{n=1}^N \langle e_n,h_k\rangle n^{iy}\langle g_k,e_{n}\rangle\\
&=&\sum_{k=0}^\infty \langle M_y(S_N^+(g_k)),h_k\rangle,
\end{eqnarray*}
where $S_N^+ g$ corresponds to the (positive) partial sum of the Fourier series defined by
\[
S_N^+g=\sum_{n=1}^N \langle g,e_{n}\rangle e_n,\qquad g\in L^p(0,1),
\]
and $M_y$ is the Fourier multiplier defined in  \eqref{eq:fourier-multiplier}. According to Lemma~\ref{lem:schatten}, we know that 
\[
n^{iy} \langle V_\xi(e_n),e_{n}\rangle\sim \frac{1}{(2i\pi)^\xi}\frac{1}{n},\qquad \mbox{as }n\to\infty.
\]
Therefore we have 
\begin{eqnarray*}
\log(N) &\lesssim& \left|\sum_{n=1}^N n^{iy} \langle V_\xi(e_n),e_{n}\rangle \right| \notag \\
&=& \left|\sum_{k=0}^\infty \langle M_y(S_N^+(g_k)),h_k\rangle, \right|\notag\\
&\less & \sum_{k=0}^\infty \|M_y(S_N^+ g_k)\|_p \|h_k\|_{p'}.
\end{eqnarray*}
According to Lemma~\ref{lem:fourier-multiplier}, we know that $M_y$ is bounded on $H^p$ for $p\gtr 1$, whence we get
\begin{equation}\label{eq::nuclear-absurd2}
\log(N)\lesssim \sum_{k=0}^\infty \|S_N^+ g_k\|_p \|h_k\|_{p'} .
\end{equation}
Thanks to the classical Riesz theorem, if $p>1$, then $S_N^+$ defines a bounded operator on $L^p$ with bound independent from $N$. Hence $\|S_N^+ g_k\|_p\lesssim \|g_k\|_p$, and we get 
\[
\log(N)\lesssim \sum_{k=0}^\infty \|g_k\|_p\|h_k\|_{p'},
\]
which contradicts \eqref{eq:nuclear-absurd}, when $p>1$.

For $p=1$, we will also get a contradiction using \eqref{eq::nuclear-absurd2}. 
Indeed, we  know that $h\in L^1(0,1)\mapsto S_N^+(h)\in L^1(0,1)$ defines a bounded operator on $L^1(0,1)$ with bound $O(log(N))$.
Moreover, for every $h\in L^1(0,1)$, we claim that 
\begin{equation}\label{olog}
\|S_N^+(h)\|_1=o(log(N)).
\end{equation}
Indeed, given $\eps>0$, there exists a trigonometric polynomial $Q$ satisfying $\|h-Q\|_1\le\eps$.
This polynomial $Q$ can be written as 
$$Q=\sum_{k=-d}^dc_ne_n,$$
where $d\in\N$ and each Fourier coefficient satisfies $|c_n|\le\|Q\|_1$, so that, when $N\gtr d$, we have $\dis\|S_N^+(Q)\|_1\le d\|Q\|_1\;.$

Therefore, for some numerical constant $C>0$, we have
\begin{align*}
\frac{\|S_N^+(h)\|_1}{\log(N)}&\le\frac{\|S_N^+(h-Q)\|_1}{\log(N)}+\frac{d\|Q\|_1}{\log(N)}\\
&\le C\|h-Q\|_1+\frac{d\|Q\|_1}{\log(N)}\\
&\le C\eps+\eps,
\end{align*}
as soon as $N$ is large enough, and \eqref{olog} is proved.\smallskip

Using \eqref{eq::nuclear-absurd2}, we obtain that 
$$1\lesssim\sum_{k=0}^{+\infty}\|h_k\|_{\infty}\,\frac{\|S_N^+(g_k)\|_1}{\log(N)}.$$
But the right hand side tends to zero when $N$ goes to $\infty$ (use \eqref{olog} and \eqref{eq:nuclear-absurd} and apply the Lebesgue domination theorem). This contradiction gives our conclusion.\end{proof}

We can now complete Corollary~\ref{cor:p-integral} and give a complete characterization of absolutely $r$-summing operators on $L^p(0,1)$. In the next result, we denote by $\Pi_r(X)=\Pi_r(X,X)$, the ideal of absolutely $r$-summing operators from $X$ to itself.
\begin{theorem}\label{thm-r-summing}
Let $1\less p,r<\infty$ and $\xi\in\C_0$. 
\begin{enumerate}[(i)]
\item If $2\gtr p\gtr 1$, then we have
\[
V_\xi\in\Pi_r(L^p(0,1))\Longleftrightarrow \xi\in\C_{\frac{1}{2}}\cdot
\]
\item If $p>2$ and $1\less r\less p'$, then we have
\[
V_\xi\in\Pi_r(L^p(0,1))\Longleftrightarrow \xi\in\C_{\frac{1}{p'}}\cdot
\]
\item If $p>2$ and $p'<r\less p$, then we have
\[
V_\xi\in\Pi_r(L^p(0,1))\Longleftrightarrow  \xi\in\C_{\frac{1}{r}}\cdot
\]
\item If $p>2$ and $p<r$, then we have
\[
V_\xi\in\Pi_r(L^p(0,1))\Longleftrightarrow \xi\in\C_{\frac{1}{p}}\cdot
\]

\end{enumerate}
\end{theorem}
\begin{proof}
$(i)$ Let $1\less p\less 2$ and $r\gtr 1$. Assume first that $\xi\in\C_{\frac{1}{2}}$ and let us show that $V_\xi\in\Pi_r(L^p(0,1))$. Since $1\less p\less 2$, the canonical injections $i:L^p(0,1)\to L^1(0,1)$ and $j:L^2(0,1)\to L^p(0,1)$ are both bounded. On the other hand, since $\xi\in\C_{\frac{1}{2}}$, according to Proposition~\ref{bound}, the operator $V_\xi:L^1(0,1)\to L^2(0,1)$ is bounded. In particular, we have the following commutative diagram 

\def\commutatif{\ar@{}[rd]|{\circlearrowleft}}
\[ 
\xymatrix{
 L^p(0,1)\ar[r]^{V_\xi} \ar[d]_i  \commutatif & L^p(0,1)\\
L^1(0,1)\ar[r]_{V_\xi} & L^2(0,1)\ar[u]_j
}
\]
Now, using the Grothendieck theorem (Theorem~\ref{Diestel-theorem4}), the operator $V_\xi:L^1(0,1)\to L^2(0,1)$ is absolutely $1$-summing, which implies by the ideal property that $V_\xi\in\Pi_1(L^p(0,1))$, {\sl a fortiori} $V_\xi\in\Pi_r(L^p(0,1))$. 

Conversely, assume that $V_\xi\in\Pi_r(L^p(0,1))$ for some $r\gtr 1$. Then, since $1\less p\less 2$ and $L^p(0,1)$ has cotype $2$, $V_\xi\in\Pi_2(L^p(0,1))$. Since the composition of two $2$-summing operators is nuclear (see Theorem~\ref{Diestel-theorem9} $(iii)$), we get that $V_{2\xi}=V_\xi\circ V_\xi$ is nuclear on $L^p(0,1)$. According to Theorem~\ref{RLnuclear}, it follows that $2\xi\in\C_1$, which means that $\xi\in\C_{\frac{1}{2}}$. That concludes the proof of $(i)$. 

$(ii)$ Let $p>2$ and $1\less r\less p'$. 
Assume first that $\xi\in\C_{\frac{1}{p'}}$ and let us show that $V_\xi\in\Pi_r(L^p(0,1))$. 
We can write $\xi=\alpha+\beta$, with $\alpha\in\C_{\frac{1}{2}}$ and $\beta\in\C_{\frac{1}{2}-\frac{1}{p}}$. 
Now according to Proposition~\ref{bound}, the operator $V_\beta:L^2(0,1)\to L^p(0,1)$ is bounded, and since $p>2$, the injection $i:L^p(0,1)\to L^2(0,1)$ is also bounded. 
Hence, we have the following commutative diagram
\def\commutatif{\ar@{}[rd]|{\circlearrowleft}}
\[ 
\xymatrix{
 L^p(0,1)\ar[r]^{V_\xi} \ar[d]_i  \commutatif & L^p(0,1)\\
L^2(0,1)\ar[r]_{V_\alpha} & L^2(0,1)\ar[u]_{V_\beta}
}
\]
Since $\alpha\in\C_{\frac{1}{2}}$, it follows from Proposition~\ref{RLHS} that $V_\alpha\in\mathcal S^2(L^2(0,1))$. 
But since $L^2(0,1)$ is a Hilbert space, then we know from \eqref{Diestel-eq6} and \eqref{Diestel-eq10} that $\mathcal S^2(L^2(0,1))=\Pi_2(L^2(0,1))=\Pi_1(L^2(0,1))$. 
Thus $V_\alpha\in\Pi_1(L^p(0,1))$ and by the ideal property, we obtain that the operator  $V_\xi$ belongs to $\Pi_1(L^p(0,1))$. 
But since $r\gtr 1$, we have the inclusion $\Pi_1(L^p(0,1))\subset \Pi_r(L^p(0,1))$, and we conclude that $V_\xi\in\Pi_r(L^p(0,1))$. \

Conversely, assume that $V_\xi\in\Pi_r(L^p(0,1))$. Then since $r\less p'$, we also have $V_\xi\in\Pi_{p'}(L^p(0,1))$. 
The operator $V_\xi^*:L^{p'}(0,1)\to L^{p'}(0,1)$ has thus an adjoint which is absolutely $p'$-summing. It thus follows from Theorem~\ref{Diestel-theorem8} $(iii)$ that $V_\xi^*$ is order bounded on $L^{p'}(0,1)$. Then, according to Proposition~\ref{OB}, we deduce that $\xi\in\C_{\frac{1}{p'}}$.\
 
(iii) Let $p>2$ and $p'<r\less p$. 
Assume first that $\xi\in\C_{\frac{1}{r}}$ and let us show that $V_\xi\in\Pi_r(L^p(0,1))$. 
We can write $\xi=\alpha+\beta$, with $\alpha\in\C_{\frac{1}{p}}$ and $\beta\in\C_{\frac{1}{r}-\frac{1}{p}}$. 
According to Proposition~\ref{bound}, the operator $V_\beta:L^r(0,1)\to L^p(0,1)$ is bounded and Proposition~\ref{bound} implies that the operator $V_\alpha:L^p(0,1)\to C([0,1])$ is bounded. 
If $i_r$ is the canonical injection from $C([0,1])$ to $L^r(0,1)$, we have the following commutative diagram 
\def\commutatif{\ar@{}[rd]|{\circlearrowleft}}
\begin{equation}\label{eq:integral-dfssdsds}
\xymatrix{
 L^p(0,1)\ar[r]^{V_\xi} \ar[d]_{V_\alpha} \commutatif & L^p(0,1)\\
C([0,1]) \ar[r]_{i_r} & L^r(0,1)\ar[u]_{V_\beta}
}
\end{equation}
Now, as already observed, the map $i_r$ is an absolutely $r$-summing operator, and then by the ideal property of $r$-summing operators, we deduce that $V_\xi\in\Pi_r(L^p(0,1))$. 

Conversely, assume that $V_\xi\in\Pi_r(L^p(0,1))$ and let us show that $\Re(\xi)> \frac{1}{r}$. 
Observe that $r'<p$ and let us take $s=\frac{1}{r'}-\frac{1}{p}=\frac{1}{p'}-\frac{1}{r}$.  
According to Proposition~\ref{bound}, the operator $V_s:L^{r'}(0,1)\to L^p(0,1)$ is bounded. Then we have the following commutative diagram 
\[
\xymatrix{
    L^{p'}(0,1) \ar[r]^{V^*_{\xi+s}}\ar[d]_{V_\xi^*} & L^r(0,1) \\
     L^{p'}(0,1)\ar[ru]_{V_s^*} & 
  }
  \]
 Since $V_\xi^*:L^{p'}(0,1)\to L^{p'}(0,1)$ has an adjoint which is absolutely $r$-summing, then, according to Theorem~\ref{Diestel-theorem8} $(ii)$, the operator $V_s^*V_\xi^*$ must be order bounded.  
 Hence $V_{\xi+s}^*:L^{p'}(0,1)\to L^r(0,1)$ is order bounded. 
 It now follows from Proposition~\ref{OB} that $\Re(\xi)+s>\frac{1}{p'}$, equivalently $\Re(\xi)>\frac{1}{r}\,\cdot$ 

$(iv)$ Let $p>2$ and $p<r$. Assume first that $\xi\in\C_{\frac{1}{p}}$ and let us show that $V_\xi\in\Pi_r(L^p(0,1))$. According to $(iii)$ with $r=p$, we know that $V_\xi\in\Pi_p(L^p(0,1))$. But since $p<r$, it follows from \eqref{Diestel-eq2} that $\Pi_p(L^p(0,1))\subset \Pi_r(L^p(0,1))$, and we immediately get that $V_\xi$ is an absolutely $r$-summing operator on $L^p(0,1)$. \

Conversely, assume that $V_\xi\in\Pi_r(L^p(0,1))$ and let us show that $\Re(\xi)>\frac{1}{p}$.  
We define $s=\frac{1}{2}-\frac{1}{p}$ and, since $p>2$, we know from Proposition~\ref{bound} that the operator $V_s:L^2(0,1)\to L^p(0,1)$ is bounded.
Moreover the canonical injection $i:L^p(0,1)\to L^2(0,1)$ is also bounded. Let us consider the following commutative diagram 
\def\commutatif{\ar@{}[d]|{\circlearrowleft}}
\[
\xymatrix{
 L^2(0,1)\ar[rr]^{V_{\xi+s}} \ar[d]_{V_s} & \commutatif &L^2(0,1)\\
L^p(0,1) \ar[rr]_{V_\xi} && L^p(0,1)\ar[u]_{i}
}
\]
By the ideal property, we deduce that $V_{\xi+s}\in\Pi_r(L^2(0,1))$. But since $L^2(0,1)$ is an Hilbert space, according to \eqref{Diestel-eq6} and \eqref{Diestel-eq10}, we have $\Pi_r(L^2(0,1))=\mathcal S^2(L^2(0,1))$.  Thus $V_{\xi+s}$ is Hilbert-Schmidt  on $L^2(0,1)$ and Proposition~\ref{RLHS} implies that $\xi+s\in\C_{\frac{1}{2}}$, equivalently $\Re(\xi)>\frac{1}{p}\cdot$\end{proof} 


\begin{rem}
\rm{
If $p>2$, $p'\less r\less p$ and $\xi\in\C_{\frac{1}{r}}$, then it follows from the proof of Theorem~\ref{thm-r-summing} that the operator $V_\xi:L^p(0,1)\to L^p(0,1)$ is not only absolutely $r$-summing but indeed $r$-integral. Indeed, this part of the proof of Theorem~\ref{thm-r-summing} works also for $r=p'$ and the conclusion follows from \eqref{eq:integral-dfssdsds} and the definition of $r$-integral operators.}

\end{rem}

\subsection{Characterization of nuclear, $p$-integral and absolutely $p$-summing operators on $C(0,1)$.}

We first show that Theorem~\ref{RLnuclear} can be extended to the case when $X=C([0,1])$.
\begin{theorem}
Let $\xi\in\mathbb C_0$.
The following assertions are equivalent:
\begin{enumerate}[(i)]
\item The operator $V_\xi:C([0,1])\to C_0([0,1])$ is nuclear. 
\item $\xi\in\C_1$.
\end{enumerate}
\end{theorem}

\begin{proof}
$(ii)\implies (i)$: Let $\xi\in\C_1$. Then $\Re(\xi/2)>\frac{1}{2}$ and the operator $V_{\xi/2}$ factorizes through the formal identity $i_2$ from $C([0,1])$ to $L^2(0,1)$ with the following commutative diagram
\[
\xymatrix{
    C([0,1]) \ar[r]^{V_{\xi/2}}\ar[d]_{i_2} &C_0([0,1]) \\
    L^2(0,1) \ar[ru]_{V_{\xi/2}} & 
  }
  \]
According to Proposition~\ref{bound} $(iv)$, $V_{\xi/2}:L^2(0,1)\to C_0([0,1])$ is bounded. On the other hand, by \eqref{Diestel-eq3}, the operator $i_2:C([0,1])\to L^2(0,1)$ is absolutely $2$-summing. Thus 
\[
V_{\frac{\xi}{2}}=V_{\frac{\xi}{2}}\circ i_2:C([0,1])\to C_0([0,1])
\]
is absolutely $2$-summing. But, according to Theorem~\ref{Diestel-theorem9} $(iii)$, the composition of two absolutely $2$-summing operators is nuclear. Therefore, denoting by $j:C_0([0,1])\to C([0,1])$ the canonical inclusion,  it follows  that $V_\xi=V_{\xi/2}\circ j\circ V_{\xi/2}$ is nuclear from $C([0,1])$ into $C_0([0,1])$.

$(i)\implies (ii)$:  According to the semigroup property and the ideal property of the class of nuclear operators, it is sufficient to check that for $\xi=1+iy$, $y\in\mathbb R$, the operator $V_\xi:C([0,1])\to C_0([0,1])$ is not nuclear. We use reductio ad absurdum: we assume that there exist two sequences $(\mu_k)_{k\gtr 0}$ in $C'([0,1])=\mathcal M([0,1])$ and $(g_k)_{k\gtr 0}$ in $C_0([0,1])$ such that 
\begin{equation}\label{eq:nuclear-absurd0}
V_\xi(f)=\sum_{k=0}^\infty \langle f,\mu_k\rangle g_k,\qquad f\in C([0,1]),
\end{equation} 
and 
\begin{equation}\label{eq:nuclear-absurd}
\sum_{k=0}^{+\infty}\|\mu_k\|\,\|g_k\|_\infty<\infty.
\end{equation}
Here $\mathcal M([0,1])$ denotes the space of  regular complex Borel measures on $[0,1]$ which identifies with the dual of $C([0,1])$ with the following duality bracket 
\[
\langle f,\mu_k\rangle=\int_0^1 f(x)\,d\overline{\mu_k}(x),\qquad f\in C([0,1]).
\]
For $n,m\in\mathbb N^*$, define
\[
f_{n,m}(u)=\begin{cases}
(1-u)^{-iy}(1-u^m)e_n(u)&\mbox{ if }u\in [0,1)\\
0&\mbox{ if }u=1
\end{cases}\]
and 
\[
f_{n}(u)=\begin{cases}
(1-u)^{-iy}e_n(u)&\mbox{ if }u\in [0,1)\\
0&\mbox{ if }u=1
\end{cases}
\]
Then $f_{m,n}\in C([0,1])$ with $\|f_{n,m}\|_\infty\leq 1$, and for every $u\in [0,1]$ we have $f_{n,m}(u)\to f_n(u)$ as $m\to\infty$.  

Using the dominated Lebesgue convergence theorem, we have,  for every $k\gtr 0$, $\lim_{m\to \infty}\langle f_{n,m},\mu_k\rangle=\langle f_n,\mu_k\rangle$. Moreover, since for every $x\in [0,1]$, we have
\[
|\langle f_{n,m},\mu_k\rangle||g_k(x)| \less \|f_{n,m}\|_{\infty}\|\mu_k\| \|g_k\|_\infty\less \|\mu_k\| \|g_k\|_\infty,
\]
it follows from \eqref{eq:nuclear-absurd} and the dominated Lebesgue convergence theorem that 
\begin{equation}\label{eq:823ZRS2}
\lim_{m\to\infty}\sum_{k=0}^\infty \langle f_{n,m},\mu_k\rangle g_k(x)=\sum_{k=0}^\infty \langle f_{n},\mu_k\rangle g_k(x),\qquad x\in [0,1].
\end{equation}
On the other hand, using one more time the dominated Lebesgue convergence theorem, we also have
\begin{equation}\label{eq:823ZRS1ZE2}
\lim_{m\to\infty}V_{\xi}(f_{n,m})(x)=V_\xi(f_n)(x).
\end{equation} 
Therefore using \eqref{eq:nuclear-absurd0}, \eqref{eq:823ZRS2} and \eqref{eq:823ZRS1ZE2}, we have, for every $x\in [0,1]$,
\begin{equation}\label{eq:formule-Vxi-fn-ponctuel}
V_\xi(f_n)(x)=\sum_{k=0}^\infty \langle f_n,\mu_k\rangle g_k(x),
\end{equation}
and the convergence is in $C_0([0,1])$ because of \eqref{eq:nuclear-absurd}. In particular, we have
\begin{equation}\label{eq:nucleaire-non-continu4ER}
\langle V_\xi(f_{n}),e_n\rangle=\sum_{k=0}^\infty  \langle f_n,\mu_k\rangle \langle g_k,e_n\rangle.
\end{equation}
Now observe that 
\[
\Gamma(\xi)\langle V_\xi(f_{n}),e_n\rangle=\int_0^1\left(\int_0^x (x-u)^{iy}(1-u)^{-iy}e^{2i\pi nu}\,du\right)e^{-2i\pi nx}\,dx.
\]
Making the change of variable $s=x-u$ and use the Fubini theorem to get
\begin{eqnarray*}
\Gamma(\xi)\langle V_\xi(f_{n}),e_n\rangle&=&\int_0^1 s^{iy}e^{-2i\pi ns}\left(\int_s^1 (1-x+s)^{-iy}\,dx\right)\,ds\\
&=&\frac{1}{1-iy}\int_0^1 s^{iy}e^{-2i\pi ns}(1-s^{1-iy})\,ds\\
&=&\frac{1}{1-iy}\left(\widehat{\psi_{iy}}(n)-\widehat{\psi_1}(n)\right),
\end{eqnarray*}
where, as in the proof of Lemma~\ref{lem:schatten},  $\psi_{\alpha}(s)=s^\alpha$, $s\in (0,1)$, with $\Re(\alpha)\in (-1,1]$. Using \eqref{eq:coeff-Fourier-psi-alpha}, we then obtain
\begin{eqnarray*}
\Gamma(\xi)\langle V_\xi(f_{n}),e_n\rangle&=&\frac{1}{1-iy}\left(\frac{\Gamma(\xi)}{(2i\pi n)^{1+iy}}+\frac{i}{2\pi n}-\frac{i}{2\pi n}+o\left(\frac{1}{n}\right)\right)\\
&=&\frac{1}{1-iy}\frac{\Gamma(\xi)}{(2i\pi n)^{1+iy}}+o\left(\frac{1}{n}\right).
\end{eqnarray*}
Therefore
\begin{equation}\label{eq:asymptotique-cas-continue}
\langle V_\xi(f_n),e_{n}\rangle\sim \frac{1}{1-iy} \frac{1}{(2i\pi n)^{1+iy}},\qquad \mbox{as }n\to\infty.
\end{equation}

The idea now to get a contradiction is to follow the proof of Theorem~\ref{RLnuclear} (in the case $p=1$) but the problem is to replace in \eqref{eq:nucleaire-non-continu4ER} the measures $\mu_k$ by functions in $L^1(0,1)$. For this purpose, define now $\widetilde{g_k}(x)=g_k(x)-xg_k(1)$, $x\in [0,1]$. We easily see that $\widetilde{g_k}\in C([0,1])$, $\widetilde{g_k}(0)=g_k(0)=0$ since $g_k\in C_0([0,1])$, and $\widetilde{g_k}(1)=g_k(1)-g_k(1)=0$. Then, it follows from Theorem~\ref{thm:salem} that there are $F_k\in L^1(0,1)$ and $G_k\in C([0,1])$ such that 
\begin{equation}\label{eq:fact-preuve-nucleaire}
\widetilde{g_k}=F_k\ast G_k,
\end{equation}
with
\begin{equation}\label{eq2:fact-preuve-nucleaire}
\|G_k\|_\infty\less 2 \|\widetilde{g_k}\|_\infty\less 4 \|g_k\|_\infty,\quad\mbox{and}\quad \|F_k\|_1=1.
\end{equation}
Moreover, note that, for $n\gtr 1$, we have  
\[
V_\xi(f_n)(1)=\int_0^1 (1-u)^{iy}(1-u)^{-iy}e^{2i\pi nu}\,du=0,
\]
and it follows from \eqref{eq:formule-Vxi-fn-ponctuel} that 
\[
\sum_{k=0}^\infty \langle f_n,\mu_k\rangle g_k(1)=0.
\]
Therefore, for every $x\in [0,1]$, we have
\[
V_\xi(f_n)(x)=\sum_{k=0}^\infty \langle f_n,\mu_k\rangle \widetilde{g_k}(x),
\]
and since $\|\widetilde{g_k}\|_\infty\less 2\|g_k\|_\infty$, it follows from \eqref{eq:nuclear-absurd} that the convergence is in $C([0,1])$. In particular, we can write
\[
\langle V_\xi(f_n),e_n\rangle=\sum_{k=0}^\infty \langle f_n,\mu_k\rangle \langle \widetilde{g_k},e_n\rangle.
\]
Denote by $d\widetilde{\mu_k}(u)=(1-u)^{iy}\chi_{[0,1)}(u)d\mu_k(u)$. Then $\langle f_n,\mu_k\rangle=\langle e_n,\widetilde{\mu_k}\rangle$, and by \eqref{eq:fact-preuve-nucleaire}, we have
\begin{eqnarray*}
\langle  f_n,\mu_k \rangle  \langle \widetilde{g_k},e_n\rangle&=&\langle e_n,\widetilde{\mu_k}\rangle \widehat{\widetilde{g_k}}(n)\\
&=& \langle e_n,\widetilde{\mu_k}\rangle \widehat{F_k}(n)\widehat{G_k}(n).
\end{eqnarray*}
Observe now that 
\begin{eqnarray*}
 \langle e_n,\widetilde{\mu_k}\rangle \widehat{F_k}(n)&=&\int_0^1 e^{2i\pi nu}\,d\overline{\widetilde{\mu_k}}(u) \int_0^1 F_k(v)e^{-2i\pi nv}\,dv\\
 &=& \int_0^1 \left(\int_0^1 F_k(v)e^{2i\pi n(u-v)}\,dv\right)\,d\overline{\widetilde{\mu_k}}(u)\\
 &=&\int_0^1\left( \int_0^u F_k(u-s)e^{2i\pi ns}\,ds\right)\,d\overline{\widetilde{\mu_k}}(u)\\
&=&\langle e_n,H_k\rangle,
\end{eqnarray*}
where 
\[
H_k(s)=\int_s^1 \overline{F_k(u-s)}\,d\widetilde{\mu_k}(u).
\]
It is easy to check that $H_k\in L^1(0,1)$ with $\|H_k\|_1\less \|F_k\|_1 \| \mu_k\|=\|\mu_k\|$.  Therefore, we have 
\begin{equation}\label{eq:sdqsqsdqsd1237Zlkn}
\langle V_\xi(f_n),e_n\rangle=\sum_{k=0}^\infty \langle e_n,H_k\rangle \langle G_k,e_n \rangle,
\end{equation}
and, according to \eqref{eq2:fact-preuve-nucleaire} and \eqref{eq:nuclear-absurd}, we have
\begin{equation}\label{eq:sdqsqsdqsd1237Zlsdskn}
\sum_{k=0}^\infty \|H_k\|_1 \|G_k\|_\infty\less 4 \sum_{k=0}^\infty \|\mu_k\| \|g_k\|_\infty<\infty.
\end{equation}
Using \eqref{eq:asymptotique-cas-continue}, \eqref{eq:sdqsqsdqsd1237Zlkn} and \eqref{eq:sdqsqsdqsd1237Zlsdskn}, we obtain now a contradiction using the same argument as in the proof of Theorem~\ref{RLnuclear} (in the case $p=1$).

 \end{proof}

As far as concerns the membership to the ideal of absolutely $r$-summing operators, the situation of $r=1$ brings to light a new phenomenon.
\begin{theorem}
Let $r\gtr 1$ and let $\xi\in\C_0$.
\begin{enumerate}[(a)]
\item If $r>1$, then the following are equivalent:
\begin{enumerate}[(i)]
\item The operator $V_\xi:C([0,1])\to C_0([0,1])$ is $r$-integral.
\item The operator $V_\xi:C([0,1])\to C_0([0,1])$ is absolutely $r$-summing.
\item $\xi\in\C_{\frac{1}{r}}$
\end{enumerate}
\item If $r=1$, then the following are equivalent:
\begin{enumerate}[(i)]
\item The operator $V_\xi:C([0,1])\to C_0([0,1])$ is $1$-integral.
\item The operator $V_\xi:C([0,1])\to C_0([0,1])$ is absolutely $1$-summing.
\item $\xi\in\overline{\C_1}$.
\end{enumerate}
\end{enumerate}
\end{theorem}

\begin{proof}
$(a)$ Let $r>1$. 

$(iii)\implies (i)$: Assume that $\xi\in\C_{\frac{1}{r}}$. Thanks to Proposition~\ref{bound} $(iv)$, the operator $V_{\xi}:L^r(0,1)\to C_0([0,1])$ is bounded. In particular, we have the following commutative diagram
\[
\xymatrix{
    C([0,1]) \ar[r]^{V_{\xi}}\ar[d]_{i_r} &C_0([0,1]) \\
    L^r(0,1) \ar[ru]_{V_{\xi}} & 
  }
  \]
Since, by \eqref{Diestel-eq345EE34}, the operator $i_r:C([0,1])\to L^r(0,1)$ is absolutely $r$-integral, we deduce that the operator $V_\xi:C([0,1])\to C_0([0,1])$ is also $r$-integral. 

$(i)\implies (ii)$: Apply \eqref{Diestel-eq7}.

$(ii)\implies (iii)$: According to the semigroup property and the ideal property of the class of absolutely $r$-summing operators, it is sufficient to check that for $\xi=\frac{1}{r}+iy$, $y\in\mathbb R$, the operator $V_\xi:C([0,1])\to C_0([0,1])$ is not nuclear. We use reductio ad absurdum and we assume that this operator is nuclear. According to Pietsch's theorem (see Theorem~\ref{Pietsch-theorem}), there exists a regular Borel probability measure $\nu$ on $[0,1]$ such that, for every $f\in C([0,1])$, we have 
\begin{equation}\label{eq:pietsch-domination}
\|V_\xi(f)\|_\infty \less \pi_r(V_\xi) \left(\int_0^1 |f(u)|^r\,d\nu(u)\right)^{1/r}.
\end{equation}
Fix $0<\varepsilon<\frac{1}{r}$ and $x\in (0,1)$. Consider also $0<a<a'<a_0<x$ and define the function $f=f_{a,a',x}$ on $(0,1)$ by 
\[
f(u)=\begin{cases}
(x-u)^{-\frac{1}{r}+\varepsilon-iy}&\mbox{if }u\in (0,a)\\
(x-u)^{-iy}(\alpha u+\beta)& \mbox{if }u\in (a,a')\\
0 & \mbox{if }u\in (a',1),
\end{cases}
\]
where 
\[
\alpha=\frac{1}{a-a'}(x-a)^{-\frac{1}{r}+\varepsilon} \quad\mbox{and}\quad \beta=-\alpha a'.
\] 
It is easy to see that $f$ is continuous on $(0,1)$ and $|f(u)|\leq |x-u|^{-\frac{1}{r}+\varepsilon}$. 
According to \eqref{eq:pietsch-domination}, we have 
\[
|V_\xi(f)(x)|\less  \pi_r(V_\xi) \left(\int_0^{a'} |f(u)|^r\,d\nu(u)\right)^{1/r}.
\]
Observe that 
\[
V_\xi(f)(x)=\frac{1}{\Gamma(\xi)}\int_0^{a'}(x-s)^{\frac{1}{r}+iy-1}f(s)\,ds,
\]
whence 
\begin{eqnarray*}
\left| \int_0^{a'}(x-s)^{\frac{1}{r}+iy-1}f(s)\,ds \right| &\lesssim &\left(\int_0^{a'}|f(u)|^r\,d\nu(u)\right)^{1/r}\\
&\less & \left(\int_0^{x} (x-u)^{-1+r\varepsilon}\,d\nu(u)\right)^{1/r}.
\end{eqnarray*}
Since 
\[
\lim_{\substack{a'\to a\\
>}}\chi_{(0,a')}(s)(x-s)^{\frac{1}{r}+iy-1}f(s)=\chi_{(0,a)}(s)(x-s)^{\varepsilon-1},
\]
and 
\[
|\chi_{(0,a')}(s)(x-s)^{\frac{1}{r}+iy-1}f(s)|\leq \chi_{(0,a_0)}(s)|x-s|^{\varepsilon-1},
\]
we can apply the dominated Lebesgue convergence theorem to obtain

\[
\frac{x^\varepsilon-(x-a)^{\varepsilon}}{\varepsilon}\lesssim  \left(\int_0^{x} (x-u)^{-1+r\varepsilon}\,d\nu(u)\right)^{1/r}.
\]
We now let $a$ goes to $x$ to get
\[
\frac{x^{r\varepsilon}}{\varepsilon^r}\lesssim \int_0^{x} (x-u)^{-1+r\varepsilon}\,d\nu(u).
\]
Since this inequality is satisfied for any $x\in (0,1)$, we integrate it and use the Fubini--Tonelli theorem for the second integral. Thus we have
\begin{eqnarray*}
\frac{1}{(1+r\varepsilon)\varepsilon^r}&\lesssim& \int_0^1\left(\int_u^1 (x-u)^{-1+r\varepsilon}\,dx\right)\,d\nu(u)\\
&=&\frac{1}{r\varepsilon}\int_0^1 (1-u)^{r\varepsilon}\,d\nu(u)\\
&\leq & \frac{1}{r\varepsilon},
\end{eqnarray*}
where the last inequality follows from the fact that $(1-u)^{r\varepsilon}\less 1$ and $\nu$ is a probability measure. Since $r>1$, we get the desired contradiction letting $\varepsilon$ goes to $0$. 

$(b)$ Let $r=1$.

$(ii)\implies (iii)$: Assume that $V_\xi:C([0,1])\to C_0([0,1])$ is absolutely $1$-summing. According to \eqref{Diestel-eq2}, for every $r>1$, the operator $V_\xi:C([0,1])\to C_0([0,1])$ is absolutely $r$-summing. If follows from $(a)$ that $\Re(\xi)>\frac{1}{r}$. Letting now $r\to 1$ implies that $\Re(\xi)\gtr 1$. 

$(i)\implies (ii)$: Apply \eqref{Diestel-eq7}.

$(iii)\implies (i)$: Assume that $\xi\in\overline{\C_1}$. According to Proposition~\ref{bound} $(v)$, the operator $V_{\xi}:L^1(0,1)\to C_0([0,1])$ is bounded. In particular, we have the following commutative diagram
\[
\xymatrix{
    C([0,1]) \ar[r]^{V_{\xi}}\ar[d]_{i_1} &C_0([0,1]) \\
    L^1(0,1) \ar[ru]_{V_{\xi}} & 
  }
  \]
Since, by  \eqref{Diestel-eq345EE34}, the operator $i_1:C([0,1])\to L^1(0,1)$ is $1$-integral, we deduce that the operator $V_\xi:C([0,1])\to C_0([0,1])$ is  $1$-integral. 

\end{proof}



\begin{thebibliography}{99}

\bibitem{AG}  Adell, J. and Gallardo-Guti{\'e}rrez, E., The norm of the {Riemann}--{Liouville} operator on {{\(L^{p}[0,1]\)}}: a probabilistic approach, Bull. Lond. Math. Soc., 39 (4), 565--574, 2007. 
\bibitem{Fanar} Alam I.Al , Gaillard L., Habib G., Lef\`evre P. and Maalouf F., Essential norm of Ces\`aro operators on $L^p$ and Ces\`aro spaces, J. Math. Anal. Appl. 467(2018), 1038-1065.
\bibitem{ACKS18} Arendt, W.; Chalendar, I.; Kumar, M.; Srivastava, S.. Asymptotic behaviour of the powers of composition operators on Banach spaces of holomorphic functions. Indiana Univ. Math. J. 67 (2018), no. 4, 1571-1595.
\bibitem{ACKS20} Arendt, W.; Chalendar, I.; Kumar, M.; Srivastava, S.. Powers of composition operators: asymptotic behaviour on Bergman, Dirichlet and Bloch spaces. J. Aust. Math. Soc. 108 (2020), no. 3, 289-320.
\bibitem{DF} Daly J.E. and Fridli S., Trigonometric multipliers on $H_{2\pi}$, Canad. Math. Bull., Vol. 48 (3), 2005, 370--381. 
\bibitem{DJT}  Diestel J.; Jarchow H. and Tonge A., Absolutely summing operators. Cambridge Studies in Advanced Mathematics, 43. Cambridge University Press, Cambridge, 1995. 
\bibitem{Diethelm} Diethelm, K., The Analysis of Fractional Differential Equations. An
Application-oriented Exposition Using Differential Operators of Caputo Type. Lecture Notes in Mathematics 2004. Springer-Verlag, Berlin, 2010. 
\bibitem{DS} Dunford N. and Schwartz J., Linear operators. Part II. Spectral theory. Selfadjoint operators in Hilbert space. Wiley Classics Library. John Wiley \& Sons, Inc., New York, 1988 (Reprint of the 1963 original).
\bibitem{tania} Eisner T.,  Stability of operators and operator semigroups, Oper. Theory: Adv. Appl., 209 (2010) Basel: Birkh{\"a}user. 
\bibitem{Fisher} Fisher, M., Imaginary powers of the indefinite integral, Amer. J. Math., vol. 93, 317-328, 1971.
\bibitem{FL} Fricain E. and Lef\`evre P., $L^2$-M\"untz spaces as model spaces, Complex Anal. Oper. Theory 13 (2019), no. 1, 127-139.
\bibitem{GGK} Gohberg I., Goldberg S. and Krupnik N., Traces and determinants of linear operators. Operator Theory: Advances and Applications, Birkh\"{a}user Verlag, Basel, vol. 116, 200.
\bibitem{GK67} Gohberg, I.C. and Krejn, M. G.. Theory and applications of Volterra Operators in Hilbert spaces, Translations of Mathematical Monographs, A.M.S.,  Vol. 24,  (1967).
\bibitem{GK69} Gohberg, I. and Krejn, M. G.. Introduction to the theory of linear nonselfadjoint operators. Translated from the Russian by A. Feinstein. Translations of Mathematical Monographs, Vol. 18. American Mathematical Society, Providence, RI, 1969.
 \bibitem{Had} Hadamard, J. Sur les fonctions enti\`eres, Bull. Soc. Math. Fr., 24(1896), 186-187.
 \bibitem{HL}Hardy, G. H.; Littlewood, J. E., Some properties of fractional integrals. I. Math. Z. 27 (1928), no. 1, 565-606. 
 \bibitem{Hewitt-Ross} Hewitt, E. and Ross, K. Abstract harmonic analysis. {V}ol. {II}: {S}tructure and analysis for compact groups. {A}nalysis on locally compact {A}belian groups. Die Grundlehren der mathematischen Wissenschaften. Springer Verlag, 1970.
 \bibitem{Kalisch}  Kalisch, G., On fractional integrals of purely imaginary order in $L^p$, Proc. Amer. Math. Soc., vol. 18, 136-139, 1967.
\bibitem{Katz}  Katznelson, Y., An introduction to harmonic analysis. Third edition. Cambridge Mathematical Library, 2004.
 \bibitem{Kato} Kato T., Perturbation Theory for linear operators. Reprint of the 1980 edition. Classics in Mathematics. Springer-Verlag, Berlin, 1995. xxii+619 pp.
\bibitem{Liouville} Liouville, J., M\'emoire sur le calcul des diff\'erentielles \`a indices quelconques, Journal de l'\'Ecole Polytechnique, Paris, 13: 71-162, 1832.
\bibitem{Lefevre} Lefevre P., The Volterra operator is finitely strictly singular from $L^1$ to $L^\infty$, Journal of Approximation Theory 214 (2017), 1-8. 
\bibitem{Mar} Marcinkiewicz J., Sur les multiplicateurs des séries de Fourier, Studia Mathematica, Vol. 8, no.1, 1939, 78--91.
\bibitem{Mikusinski} Mikusinski J., Une simple d\'emonstration du th\'eor\`eme de Titchmarsh sur la convolution, Bull. Acad. Polon. Sci. S\'er. Sci. Math. Astr. Phys. 7 (1959), 715--717.

\bibitem{RS} Reed M. and Simon B., Methods of modern mathematical physics. {I}, 2nd ed., Academic Press, Inc., New York, 1980.
\bibitem{Riemann} Riemann B. Versuch einer allgemeinen Auffassung der integration und differentiation, in Weber, H. (ed.), Gesammelte Mathematische Werke, Leipzig, 1896.
\bibitem{ShapiVol} Shapiro J.H., Volterra Variations, Portland 2029. 
 \bibitem{Titch} Titchmarsh E.C., The theory of functions, 2nd ed., Oxford Univ. Press, Oxford, 1939.
 \bibitem{Woj} Wojtaszczyk P., Banach spaces for analysts, Cambridge Studies in Advanced Mathematics, vol. 25, 1991.
 \bibitem{zhao} Zhao R., Generalization of Schur's Test and its Application to a Class of Integral Operators on the Unit Ball of $\C^n$,  Integral Equations and Operator Theory 82, pages 519-532 (2015).

\end{thebibliography}
\end{document}